\pdfoutput=1

\RequirePackage{snapshot}
\documentclass[12pt,fabfinal]{fabarticle}

\usepackage[OT1]{fontenc}

\usepackage[utf8]{inputenc}

\usepackage
[a4paper,tmargin=3truecm,bmargin=2.8truecm,rmargin=2.5truecm,lmargin=2.5truecm,twoside,verbose=true]{geometry}
\usepackage{amsmath,amssymb,latexsym}
\renewcommand{\emptyset}{\varnothing}

\usepackage{stmaryrd,wasysym,upgreek,mathrsfs,dsfont,mathtools}

\usepackage{xcolor}
\usepackage[pdftex]{graphicx}
\usepackage{slashed}
\usepackage{subfig}
\usepackage{calc}
\usepackage{array}
\usepackage{longtable,cellspace}

\usepackage{lscape}
\usepackage{multicol}

\allowdisplaybreaks[1]

\usepackage{tikz}
\usetikzlibrary{arrows,backgrounds,decorations.pathmorphing}

\usepackage{fabmacromath}

\usepackage{pdfsync}

\newcommand{\lrtimes}{\super{\ltimes}{\rtimes}}

\usepackage[square,numbers,sort&compress]{natbib}
\usepackage{bibentry}

\usepackage[pdftex]{hyperref}
\hypersetup{%
  colorlinks=true,%
  linkcolor=black,%
  anchorcolor=black,%
  citecolor=black,%
  urlcolor=black,%
  pdftitle={},%
  pdfauthor={Fabien Vignes-Tourneret},%
  pdfsubject={},%
  pdfkeywords={},%
  bookmarksopen=false,%
}

\usepackage[english]{babel}

\usepackage[hyperref,amsmath,thmmarks]{ntheorem}

\usepackage[english]{cleveref}
\crefname{thm}{theorem}{theorems}
\crefname{prop}{proposition}{propositions}
\crefname{cor}{corollary}{Corollary}
\crefname{defn}{definition}{Definition}
\crefname{defnb}{definition}{Definition}

\newcommand{\set}[1]{\lb #1\rb}

\title{\textsf{Non-orientable quasi-trees\\
    for the \BRp{}}}
\author{\textrm{Fabien Vignes-Tourneret}}
\date{}

\begin{document}
\nobibliography*
\maketitle

\begin{abstract}
 We extend the quasi-tree expansion of 
 A.~Champanerkar, I.~Kofman, and N.~Stoltzfus to not necessarily orientable ribbon graphs. We study the duality properties of the \BRp{} in terms of this expansion. As a corollary, we get a ``connected state'' expansion of the Kauffman bracket of virtual link diagrams. Our proofs use extensively the partial duality of S.~Chmutov.
\end{abstract}

\Kb.

\section{Introduction}
\label{sec:introduction}

Ribbon graphs are a topological generalization of graphs. They can be
described in (at least) three different ways: as embedded graphs, as
possibly non-orientable surfaces with boundary or as triples of
permutations describing the vertices, the edges and their possible twists (see \cref{RibbonEx1}). In the following, we will mainly adopt the surface point of view.

In $1954$, W.~Tutte defined a graph invariant \citep{Tutte1954aa}, now named Tutte polynomial, which is a generalization of many other invariants such as the chromatic and flow polynomials. The \Tp{} may be described either via a spanning subgraph expansion, a spanning tree expansion, or, recursively, by reduction relations. More recently, B.~Bollob\'as and O.~Riordan defined a ribbon graph invariant which generalizes the Tutte polynomial. The \BRp{} also has three different possible definitions. The present article focuses on one of them.

It turns out that, for ribbon graphs, the right topological
generalization of a spanning tree is a quasi-tree. A quasi-tree is a spanning subribbon graph with only one boundary component (or face). A.~Champanerkar, I.~Kofman, and N.~Stoltzfus proved that the \BRp{} has a quasi-tree expansion \citep{Champanerkar2007aa}. Their work was restricted to orientable ribbon graphs. Our article aims at extending their expansion to the non-orientable case.\\

Very recently, S.~Chmutov defined a generalization of the usual
Euler-Poincaré (hereafter natural) duality for ribbon graphs
\citep{Chmutov2007aa}. His partial duality consists in forming the
natural dual but only with respect to a spanning sub(ribbon)graph. We
find that this new duality is an interesting, fruitful and promising
framework for the study of ribbon graphs and their invariants. In our
opinion, the use of the partial duality simplifies the formulation of
the proofs presented in this article a lot.\\

The paper is organized as follows. In \cref{sec:part-dual-ribb}, we
recall the basic definitions of a ribbon graph and the partial
duality. The spanning tree expansion of the \Tp{} relies on a notion
of activity of an edge with respect to a spanning
tree. \Cref{sec:activities-wrt-quasi} defines the generalization of
Tutte's activities to adapt them to non-plane ribbon graphs and
quasi-trees. The spanning tree expansion of the \Tp{} consists in a
factorization of the monomials of the spanning subgraph expansion. To
this end, the subgraphs are grouped into packets, each of which is
labelled by a spanning tree. In \cref{sec:binary-tree-partial}, we group the subribbon graph of a ribbon graph into packets, naturally associated with quasi-trees. \Cref{sec:non-orientable-quasi} is devoted to the statement and proof of our main theorem, namely a quasi-tree expansion of the \BRp{} of not necessarily orientable ribbon graphs. We also give the corresponding expansion for the multivariate version of this polynomial \citep{Moffatt2008ab,Vignestourneret2008aa}. In \cref{sec:duality-properties}, we recover the duality property of the \BRp{}, namely its invariance under partial duality at $q\defi xyz^{2}=1$ \citep{Chmutov2007aa,Vignestourneret2008aa}, but in terms of its quasi-tree expansion. This allows us to get an alternative expression for this polynomial at $q=1$.

The \Kb{} of a virtual link diagram and the \BRp{} of ribbon graphs have been proven to be related to each other \citep{Chmutov2007ab,Dasbach2008aa,Chmutov2008aa,Chmutov2007aa}. As a consequence, the quasi-tree expansion of the \BRp{} allows to get such an expansion for the \Kb. In \cref{sec:kauffm-brack-virt}, we translate this expansion into pure ``knot theoretical'' terms to get a connected state (ie a one-component state) expansion of the \Kb{} of a virtual link diagram. Finally, an appendix exemplifies the quasi-tree (resp.\@ connected state) expansion of the \BRp{} (resp.\@ \Kb).
\begin{note}
  During the publishing process, the author discovered that a
  quasi-tree expansion for the (unsigned) \BRp{} of non-orientable ribbon graphs
  has been derived by Ed~Dewey \citep{Dewey2007aa}. His expansion
  is true for any $w$ but does not make use of Chmutov's partial duality.
\end{note}

\section{Partial duality of a ribbon graph}
\label{sec:part-dual-ribb}

\subsection{Ribbon graphs}
\label{sec:ribbon-graphs}

A ribbon graph $G$ is a (not necessarily orientable) surface with boundary represented as the union of two sets of closed topological discs called vertices $V(G)$ and edges $E(G)$. These sets satisfy the following:
\begin{itemize}
  \item vertices and edges intersect by disjoint line segment,
  \item each such line segment lies on the boundary of precisely one vertex and one edge,
  \item every edge contains exactly two such line segments.
\end{itemize}
\Cref{RibbonEx1} shows an example of a ribbon graph. Note that we allow the edges to twist (giving the possibility for the surfaces associated with the ribbon graphs to be non-orientable). A priori an edge may twist more than once but the \BRp{} only depends on the parity of the number of twists (this is indeed the relevant information for counting the boundary components of a ribbon graph), so that we will only consider edges with at most one twist.
\begin{figure}[!htp]
  \centering
  \subfloat[A signed ribbon graph]{{\label{RibbonEx1}}\includegraphics[scale=.5]{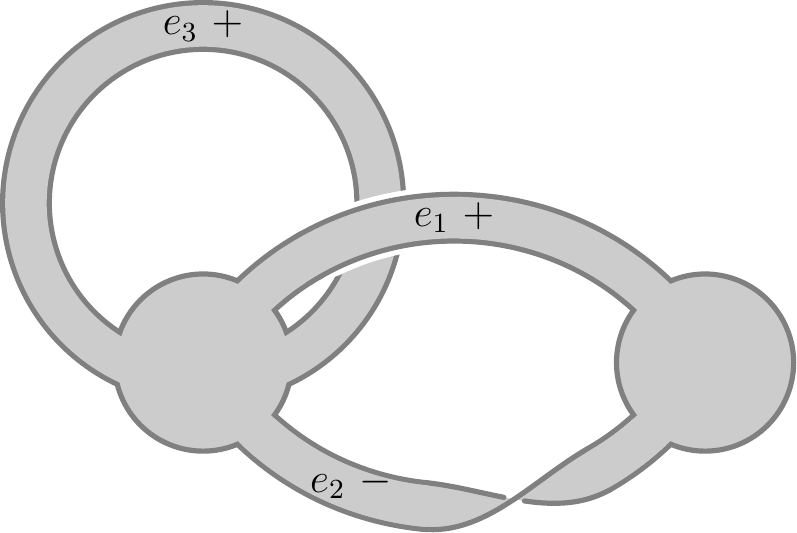}}\hspace{2cm}
  \subfloat[The combinatorial representation]{{\label{CombRep1}}\includegraphics[scale=.5]{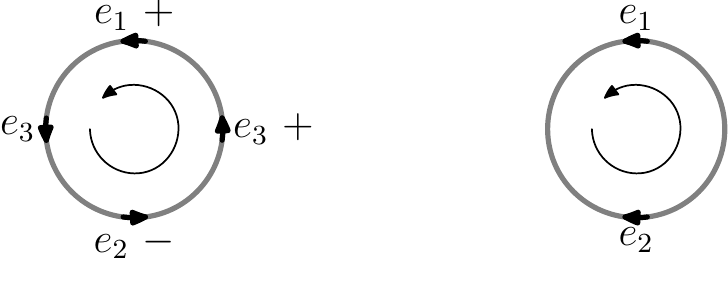}}
  \label{TwoRepEx}
  \caption{Two representations of a ribbon graph}
\end{figure}

\begin{defn}[Notations]\label{def:notations}
  Let $G$ be a ribbon graph. In the rest of this article, we will use
  the following notation:
\begin{itemize}
\item $v(G)=\card V(G)$ is the number of vertices of $G$,
\item $e(G)=\card E(G)$ is the number of edges of $G$,
\item $k(G)$ is its number of components,
\item $r(G)=v(G)-k(G)$ is its rank,
\item $n(G)=e(G)-r(G)$ is its nullity,
\item $f(G)$ is its number of boundary components (faces),
\item for all $E'\subseteq E(G)$, $F_{E'}$ is the spanning sub(ribbon)
  graph of $G$ the edge-set of which is $E'$ and
\item for all $E'\subseteq E(G)$, $E'^{c}\defi E(G)\setminus E'$.
\end{itemize}
\end{defn}

For the construction of partial dual graphs, another (equivalent)
representation of ribbon graphs will be useful. It has been introduced
in \citep{Chmutov2007aa} and will be referred to hereafter as the
``combinatorial representation''. It can be described as follows: for
any ribbon graph $G$, pick out an orientation of each vertex-disc and
each edge-disc. The orientation of the edges induces an orientation of
the line segments along which they intersect the vertices. Then draw
all vertex-discs as disjoint circles in the plane oriented
counterclockwise (say) but for the edges, draw only the arrows
corresponding to the orientation of the line segments. \Cref{CombRep1}
gives the combinatorial representation of the graph of \cref{RibbonEx1}. Each edge $e\in E(G)$ is represented as a pair of
arrows which share the same label $e$.

Given a combinatorial representation, one reconstructs the corresponding ribbon graph as follows. Each circle of the representation is filled: this gives the vertex-discs. Let us consider a couple $c_{e}$ of arrows with the same label (i.e.\@ corresponding to the same edge). These two arrows belong to the boundaries of vertices $v_{1}$ and $v_{2}$, which may be equal. One draws an edge which intersects $v_{1}$ and $v_{2}$ along the arrows of $c_{e}$. We now have to decide whether this edge twists or not. This depends on the relative direction of the two arrows. Actually there is a unique choice (twist or not) such that there exists an orientation of the edge which reproduces the couple of arrows under consideration. So we proceed as explained for each couple of arrows with a common label.

\paragraph{Loops}
\label{sec:loops}

Unlike the graphs, the ribbon graphs may contain four different kinds
of loops. A loop may be \textbf{orientable} or not, a
\textbf{non-orientable} loop being a twisting edge. Let us consider
the general situations of \cref{fig:loopRibbon}. The boxes $A$ and $B$
represent any ribbon graph such that the picture \labelcref{OrLoop} (resp.\@
\labelcref{NonOrLoop}) describes any ribbon graph $G$ with an orientable
(resp.\@ non-orientable) loop $e$ at vertex $v$. A loop is said to be \textbf{nontrivial} if there is a path in $G$ from $A$ to $B$ which does not contain $v$. If not the loop is called \textbf{trivial} \citep{Bollobas2002aa}.
\begin{figure}[!htp]
  \centering
  \subfloat[An orientable loop]{{\label{OrLoop}}\includegraphics[scale=.5]{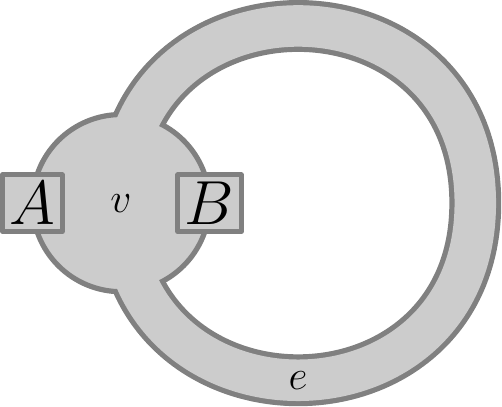}}\hspace{2cm}
  \subfloat[A non-orientable loop]{{\label{NonOrLoop}}\includegraphics[scale=.5]{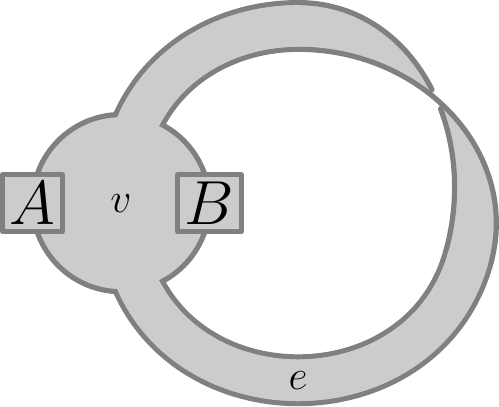}}
  \caption{Loops in ribbon graphs}
  \label{fig:loopRibbon}
\end{figure}

A ribbon graph $G$ is said to be \textbf{signed} if an element of $\{+,-\}$ is assigned to each edge. This is achieved via a function $\veps_{G}:E(G)\to\{-1,1\}$.

\subsection{Partial duality}
\label{sec:partial-duality}

S.~Chmutov introduced a new ``generalized duality'' for ribbon graphs
which generalizes the usual notion of duality (see \citep{Chmutov2007aa}). In \citep{Moffatt2008aa}, I.~Moffatt renamed this new duality as ``partial duality''. We adopt this designation here. We now describe the construction of a partial dual graph and give a few properties of the partial duality.\\

Let $G$ be a ribbon graph and $E'\subseteq E(G)$. Let $F_{E'}$ be the
spanning subribbon graph of $G$ whose edge-set is $E'$. We will
construct the dual $G^{E'}$ of $G$ with respect to the edge-set $E'$;
see \cref{fig:ribbonEx} for an example. The general idea is the
following. We consider the spanning subribbon graph $F_{E'}$ and mark it with
arrows to keep track of the edges in $E(G)\setminus E'$. Then we
take the natural dual $F_{E'}^{\star}$ of the arrow-marked ribbon
graph $F_{E'}$. Finally we use the arrows on $F_{E'}^{\star}$ to
redraw the edges in $E(G)\setminus E'$ \citep{Moffatt2009aa}.\\

We now describe the partial duality more precisely. Recall that each edge of $G$ intersects one or two vertex-discs along two line segments. In the following, each time we write ``line segment'', we mean the intersection of an edge and a vertex.

We actually construct the combinatorial representation of the partial
dual $G^{E'}$ of $G$. We first choose an orientation for each edge of
$G$. It induces an orientation of the boundaries of the edges. For
each edge in $E(G)\setminus E'$, and as was explained for the
combinatorial representation, we draw one arrow per oriented line
segment at the boundary of that edge and in the direction of the orientation. For the edges in $E'$, we proceed differently. Considering them as rectangles, they have two opposite sides that they share with one or two disc-vertices: these are the line segments defined above. But they also have two other opposite sides that we call ``long sides''. The chosen orientation induces an orientation of the long sides of the edges in $E'$; see \cref{fig:BoundComp} for an example. We draw an arrow on each long side of each edge in $E'$ according to the chosen orientation. Now draw each boundary component of $F_{E'}$ as a circle with arrows corresponding to the edges of $G$. The result is the combinatorial representation of $G^{E'}$; see \cref{fig:DualExComb,fig:DualEx}. Note that $G$ and $G^{E'}$ are generally embedded into different surfaces (they may have different genera).\\

As in the case of the natural duality, and for any $E'\subseteq E(G)$, there is a bijection between the edges of $G$ and the edges of its partial dual $G^{E'}$. Let $\phi:E(G)\to E(G^{E'})$ denote this bijection. We explain now how it is defined from the construction of the partial dual graph. As explained above, on each edge $e\in E(G)$, we draw two arrows compatible with an arbitrarily chosen orientation of this edge. If $e\in E'$, these arrows are drawn on the long sides of $e$. If $e\in E(G)\setminus E'$, they belong to the line segments along which $e$ intersects its end-vertices. Anyway we label this couple of arrows with $\phi(e)$. Proceeding like that for all edges of $G$, we build the combinatorial representation of the dual $G^{E'}$ namely we get one circle per boundary component of the spanning subribbon graph $F_{E'}$ of $G$. On each of these circles, there are arrows which represent the edges of $G^{E'}$. For each couple $c_{e'}$ of arrows that is for each edge $e'$ of $G^{E'}$, there exists a unique $e\in E(G)$ such that  $c_{e'}$ bears the label $\phi(e)$. The map $\phi$ is then clearly a bijection.\\

For signed graphs, the partial duality comes with a change of the sign
function. The function $\veps_{G^{E'}}$ is defined by the following
equations: for all $e\in E\setminus E',\,\veps_{G^{E'}}(e)=\veps_{G}(e)$ and
for all $e\in E',\,\veps_{G^{E'}}(e)=-\veps_{G}(e)$.\label{page:SignPartialDual}
\begin{figure}[!htp]
  \centering
  \subfloat[A ribbon graph $G$ with $E'=\{e_{1}\}$]{{\label{fig:ribbonEx}}\includegraphics[scale=.5]{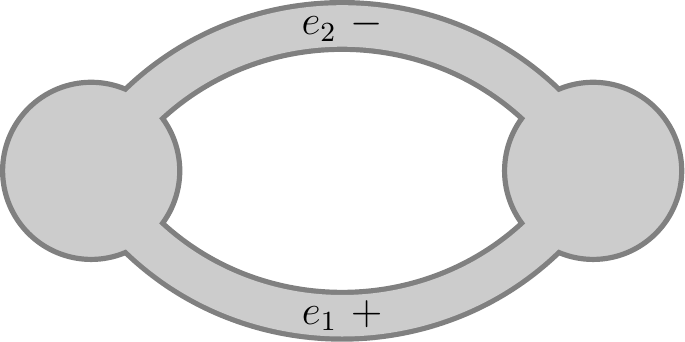}}\hspace{2cm}
  \subfloat[The combinatorial representation of $G$]{{\label{fig:ribbonExReprArrow}}\includegraphics[scale=.5]{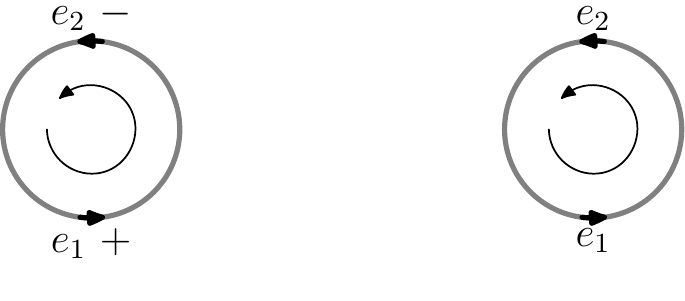}}\\
  \subfloat[The boundary component of $F_{E'}$]{{\label{fig:BoundComp}}\includegraphics[scale=.5]{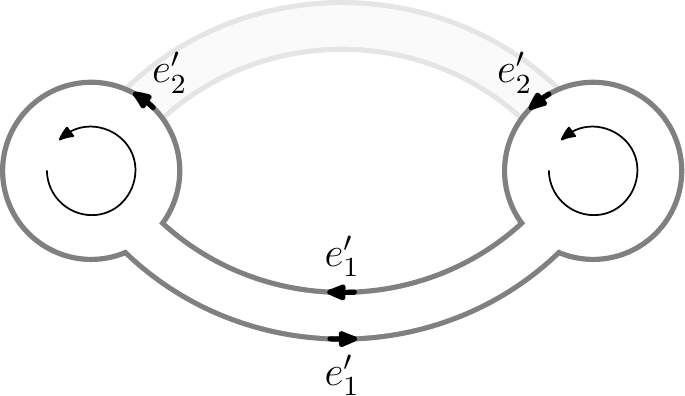}}\hspace{2cm}
  \subfloat[The combinatorial representation of $G^{E'}$]{{\label{fig:DualExComb}}\includegraphics[scale=.5]{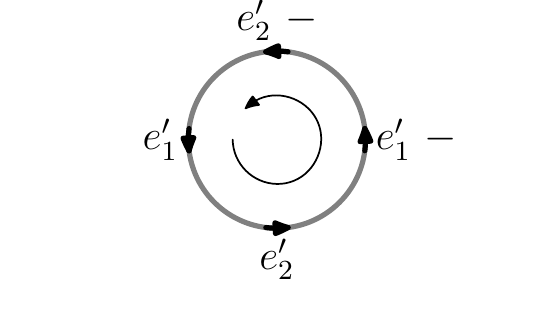}}\\
  \subfloat[The dual $G^{E'}$]{{\label{fig:DualEx}}\includegraphics[scale=.5]{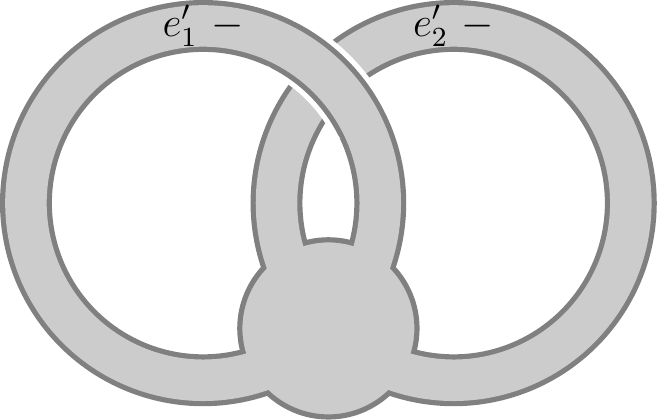}}
  \caption{Construction of a partial dual}
  \label{PartDualEx}
\end{figure}\\

S.~Chmutov proved among other things the following basic properties of the partial duality:
\begin{lemma}[\citep{Chmutov2007aa}]
  \label{lem:SimpleProp}
  For any ribbon graph $G$ and any subset of edges $E'\subseteq E(G)$, we have
  \begin{itemize}
  \item $(G^{E'})^{E'}=G$,
  \item $G^{E(G)}=G^{\star}$ and
  \item let $e\notin E'$; then $G^{E'\cup\{e\}}=(G^{E'})^{\{e\}}$.
  \end{itemize}
\end{lemma}

The partial duality allows an interesting and fruitful definition of the contraction of an edge:
\begin{defnb}[Contraction of an edge \citep{Bollobas2002aa}]\label{def:Contraction}
  Let $G$ be a ribbon graph and $e\in E(G)$ any of its edges. We define the contraction of $e$ by
  \begin{align}
    G/e\defi&G^{\{e\}}-e.\label{eq:ContractionDef}
  \end{align}
\end{defnb}
From the definition of the partial duality, one easily checks that,
for an edge incident with two different vertices, the \cref{def:Contraction} coincides with the usual intuitive contraction
of an edge. The contraction of a loop depends on its orientability;
see \cref{fig:OrLoopContraction,fig:NonOrLoopContraction}.

Different definitions of the contraction of a loop have been used in
the literature. One can define $G/e\defi G-e$. In
\citep{Huggett2007aa}, S.~Huggett and I.~Moffatt give a definition
which leads to surfaces which are no longer ribbon graphs. The \cref{def:Contraction} maintains the duality between
contraction and deletion.
\vspace{1cm}
\begin{figure}[!htp]
  \centering
  \begin{minipage}[c]{.4\linewidth}
    \centering
    \includegraphics[scale=.5]{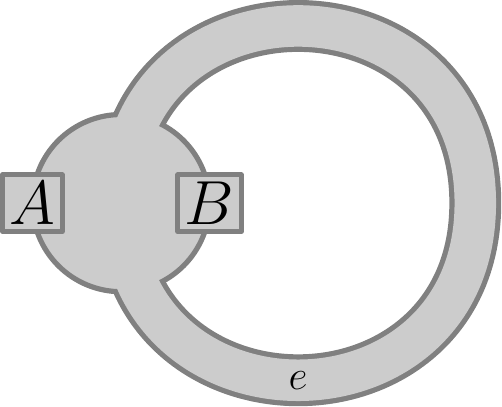}\\
    A ribbon graph $G$ with an orientable loop $e$
  \end{minipage}%
  \qquad$\longrightarrow$\qquad%
  \begin{minipage}[c]{.4\linewidth}
    \centering
    \includegraphics[scale=.5]{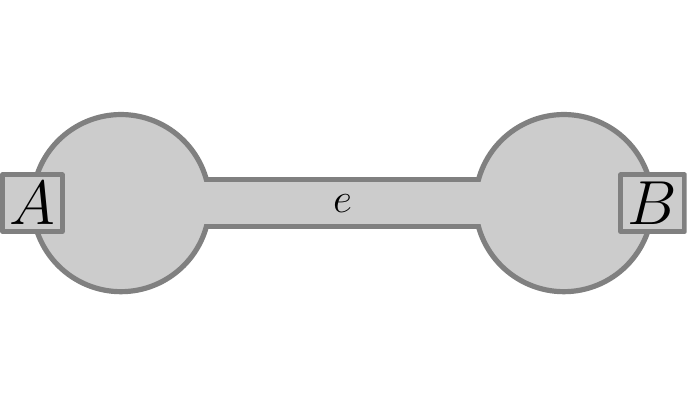}\\
    $G^{\{e\}}$\\
    \phantom{orientable loop $e$}
  \end{minipage}\\%
  $\longrightarrow$\qquad%
  \begin{minipage}[c]{.4\linewidth}
    \centering
    \includegraphics[scale=.5]{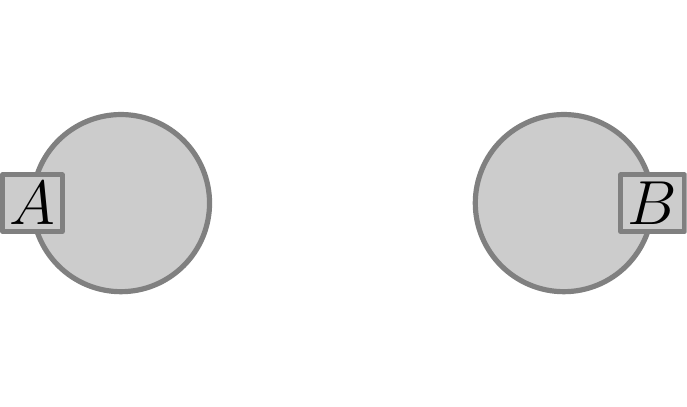}\\
    $G/e=G^{\{e\}}-e$
  \end{minipage}
  \caption{Contraction of an orientable loop}
  \label{fig:OrLoopContraction}
\end{figure}
\begin{figure}[!htp]
  \centering
  \begin{minipage}[c]{.4\linewidth}
    \centering
    \includegraphics[scale=.5]{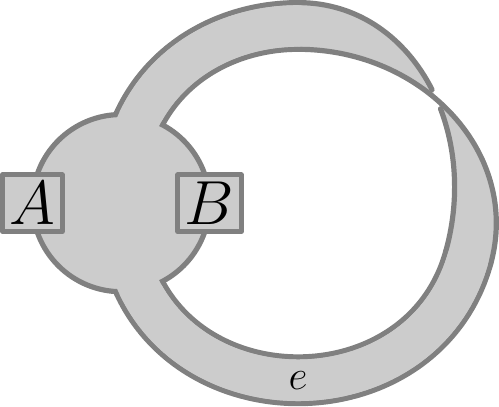}\\
    A ribbon graph $G$ with a non-orientable loop $e$
  \end{minipage}%
  \qquad$\longrightarrow$\qquad%
  \begin{minipage}[c]{.4\linewidth}
    \centering
    \includegraphics[scale=.5]{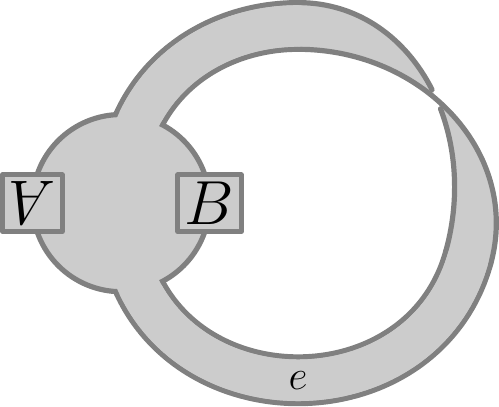}\\
    $G^{\{e\}}$\\
    \phantom{non-orientable loop $e$}
  \end{minipage}\\%
  $\longrightarrow$\qquad%
  \begin{minipage}[c]{.4\linewidth}
    \centering
    \includegraphics[scale=.5]{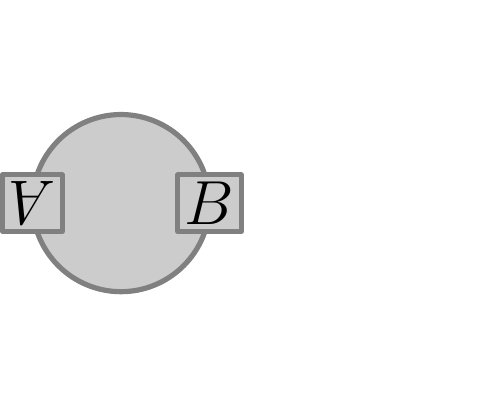}\\
    $G/e=G^{\{e\}}-e$
  \end{minipage}
  \caption{Contraction of a non-orientable loop}
  \label{fig:NonOrLoopContraction}
\end{figure}

\newpage
\section{Activities with respect to a quasi-tree}
\label{sec:activities-wrt-quasi}

\begin{defn}[Quasi-tree \citep{Dasbach2006aa,Champanerkar2007aa}]
  A quasi-tree $Q$ is a ribbon graph with $f(Q)=1$. Let $G$ be a
  ribbon graph that is not necessarily orientable. The set of spanning
  subribbon graphs of $G$ which are quasi-trees is denoted by $\mathbf{\cQ_{G}}$.
\end{defn}
A quasi-tree is a generalization of a spanning tree in the following
sense. If $G$ is a plane ribbon graph, then $\cQ_{G}$ is the set of
spanning trees of $G$. For a non-plane ribbon graph $G$, $\cQ_{G}$
contains the spanning trees of $G$ and each quasi-tree contains a spanning tree.

\begin{defn}[Crossing edges]\label{def:Cross}
  Let $G$ be a one-vertex ribbon graph. Let $e,e'\in E(G)$ be two edges of $G$. They intersect
  the vertex of $G$ along line segments $s_{1}(e),s_{2}(e),s_{1}(e')$
  and $s_{2}(e')$. The edges $e$ and $e'$ \textbf{cross} each other
  (written $e\lrtimes e'$) if, turning around the vertex of $G$ (in any direction), one meets
  the line segments of $e$ and $e'$ alternately, say $s_{1}(e),s_{1}(e'),s_{2}(e),s_{2}(e')$.
\end{defn}
For example, in \cref{fig:CrossEdges}, $e_{1}\lrtimes e_{2}$,
$e_{1}\lrtimes e_{3}$ but $e_{2}$ and $e_{3}$ do not cross each other.
\begin{figure}[!htp]
  \centering
  \subfloat[A ribbon graph
  $G$]{{\label{fig:CrossEdges}}\includegraphics[scale=.5]{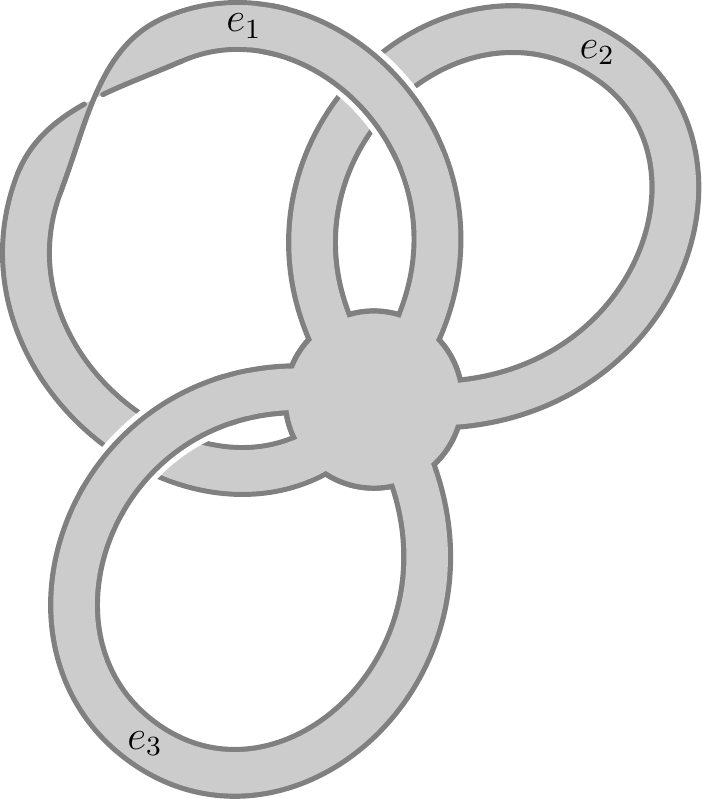}}\hspace{2cm}
  \subfloat[Its dual $\displaystyle G^{\{e_{1}\}}$]{{\label{fig:LinkEdges}}\includegraphics[scale=.5]{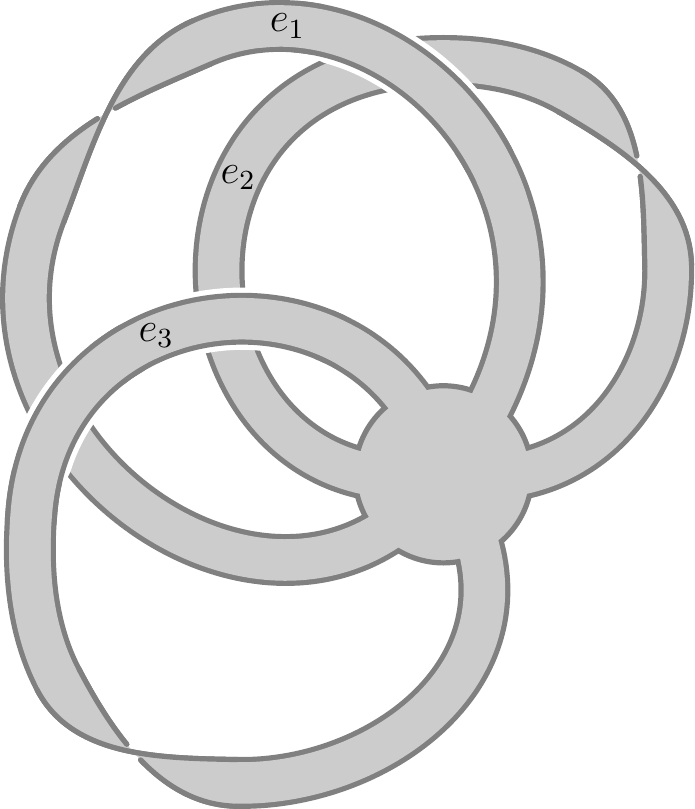}}
 \caption{Crossing and linking edges}
  \label{fig:CrossLinkEdges}
\end{figure}\\

If $Q$ is a quasi-tree of a ribbon graph $G$, the partial dual
$G^{E(Q)}$ of $G$ is a one-vertex ribbon graph.
\begin{defn}[Linking edges]\label{def:LinkingEdges}
  Let $G$ be a ribbon graph and $e,e'\in E(G)$ be two of its edges. Let
  $Q$ be a quasi-tree in $G$. We say that $e$ and $e'$ \textbf{link}
  each other (with respect to $Q$) if they cross each other in $G^{E(Q)}$.
\end{defn}
One of the quasi-trees of the ribbon graph of figure
\cref{fig:CrossEdges} is $\displaystyle F_{\{e_{1}\}}$. The edges $e_{2}$ and $e_{3}$ link each
other with respect to $F_{\{e_{1}\}}$: they cross each other in
$G^{\{e_{1}\}}$; see \cref{fig:LinkEdges}.\label{p:ONODuality}
 \begin{rem}
   In \citep{Champanerkar2007aa} the authors associated a chord diagram
   with any ribbon graph $G$ and quasi-tree $Q\in\cQ_{G}$. They defined two edges to link
   each other if their corresponding chords cross each other. This
   definition is actually the same as \cref{def:LinkingEdges} as the circle of the chord diagram in
   \citep{Champanerkar2007aa} is the boundary of the unique vertex in $G^{E(Q)}$.
 \end{rem}
 \begin{defn}[Activities with respect to a quasi-tree]\label{def:activities}
   Let $G$ be a ribbon graph and $Q\in\cQ_{G}$ one of its
   quasi-trees. Let $\prec$ be a total order on the set $E(G)$ of
   edges of $G$. An edge $e\in E(G)$ is said to be \textbf{live} if it does
   not link any lower-ordered edge; otherwise it is
   \textbf{dead}. Moreover $e$ is \textbf{internal} if it belongs to $E(Q)$ and
   \textbf{external} otherwise.

   We let $\mathbf{\cI(Q)}$ be the set of internally live edges of $G$ (with respect to
   $\prec$). Let $\mathbf{\cI_{\text{\textbf{o}}}(Q)}$ (resp.\@ $\mathbf{\cI_{\text{\textbf{n}}}(Q)}$) be the set of internally live
   edges that form orientable (resp.\@ non-orientable\footnote{As $v(G^{E(Q)})=1$, any edge of $G^{E(Q)}$ is a
   loop. In the following, when we write that an edge is orientable
   (or not) it always means ``as a loop in a certain $G^{E(Q)}$''.})
 loops in $G^{E(Q)}$. Obviously
   $\cI_{\text{o}}(Q)\cap\cI_{\text{n}}(Q)=\emptyset$ and
   $\cI(Q)=\cI_{\text{o}}(Q)\cup\cI_{\text{n}}(Q)$. We define similarly $\mathbf{\cE(Q)}$, $\mathbf{\cE_{\text{\textbf{o}}}(Q)}$ and
   $\mathbf{\cE_{\text{\textbf{n}}}(Q)}$ for the externally live
   edges.

   Finally we let $\mathbf{\cD(Q)}$ be the set of internally dead
   edges of $G$ with respect to $Q$ and $\prec$.
 \end{defn}
 One easily checks that for plane ribbon graphs, \cref{def:activities} of live (resp.\@ dead) edges coincides with the
 definition of active (resp.\@ inactive) edges in the spanning tree
 expansion of the \Tp{} \citep{Tutte1954aa}. In contrast, for non-plane ribbon
 graphs, those definitions are different. First of all there are more
 quasi-trees than spanning trees but even with respect to a spanning
 tree the activities are different. Let us once more consider the
 example of \cref{fig:CrossEdges} with $e_{1}\prec e_{2}\prec e_{3}$. The only spanning tree in $G$
 is $F_{\emptyset}$ (and $G^{\emptyset}=G$). All edges are externally
 active but $\cI=\cD=\emptyset$, $\cE=\cE_{\text{n}}=\{e_{1}\}$ and
 $e_{2},e_{3}$ are externally dead. With respect to the quasi-tree $F_{\{e_{1}\}}$, we have
$\cI=\cI_{\text{n}}=\{e_{1}\}$, $\cD=\emptyset$ and
$e_{2},e_{3}$ are externally dead.
\newpage
\section{Binary tree of partial resolutions}
\label{sec:binary-tree-partial}

Following \citep{Champanerkar2007aa}, we construct a rooted binary tree which
allows us to group the spanning subribbon graphs of a given connected ribbon graph into
packets labelled by the quasi-trees of $G$. The members of these
packets are in one-to-one correspondence with the subsets of orientable internally and
externally live edges; see \cref{lem:LiveDead}.

\subsection{Partial resolutions and duality}
\label{sec:part-resol-dual}

In this section, we prove two lemmas about resolutions and
quasi-trees. These lemmas will be useful for the proof of \cref{lem:LiveDead}. The proofs below use Chmutov's partial duality.

\begin{defn}[Resolutions]\label{def:resolutions}
  Let $G$ be a ribbon graph. A \textbf{resolution} $s$ of $G$ is a map from
  $E(G)$ into $\{0,1\}$. Each resolution determines a spanning
  subribbon graph $H_{s}$ such that $E(H_{s})\defi\lb e\in E(G)\tqs
  s(e)=1\rb$.

  A \textbf{partial resolution} $\rho$ of $G$ is a map from $E(G)$ into
  $\{0,1,*\}$. We define $H_{\rho}$ to be the spanning subribbon graph of $G$
  whose edge-set is $\lb e\in E(G)\tqs \rho(e)=1\rb$. We let $U(\rho)\defi\lb e\in E(G)\tqs \rho(e)=*\rb$ be
  the set of unresolved edges. Each partial resolution determines a subset of
  the spanning subribbon graphs of $G$:\\$[\rho]\defi\lb\text{resolutions $s$
    of $G$}\tqs s(e)=\rho(e)\text{ if }\rho(e)\in\{0,1\}\rb$.
\end{defn}
Let $F\subseteq G$ be a spanning subribbon graph of $G$. The number of faces
 of $F$ equals the number of vertices of its natural dual
 $F^{\star}$. But in the following it will be necessary to express
 this number in terms of the partial dual of $G$ with respect to $E(F)$, namely
 \begin{align}
   f(F)=&v(F^{\star})=v(G^{E(F)}).\label{eq:FacesVertexPartial}
 \end{align}
\begin{prop}\label{prop:F'vsF}
   Let $G$ be a ribbon graph and $F,F'\subseteq G$ two spanning subribbon graphs of $G$. Let $\Delta\defi\Delta(F,F')=(E(F)\cup
   E(F'))\setminus(E(F)\cap E(F'))$. Then we have
   \begin{align}
     f(F')=v\big((G^{E(F)}-\Delta^{c})/\Delta\big).
   \end{align}
 \end{prop}
 \begin{proof}
   As in \cref{eq:FacesVertexPartial},
   $f(F')=v(G^{E(F')})$. But $G^{E(F')}=\big(G^{E(F)}\big)^{\Delta}$
   so
   $f(F')=v\Big(\big(G^{E(F)}\big)^{\Delta}\Big)=v\Big(\big(G^{E(F)}\big)^{\Delta}-E(G)\Big)$. Using
   $E(G)=\Delta\cup\Delta^{c}$ and for any ribbon graph $G$ and any
   $E',E''\subseteq E(G)$ such that $E'\cap E''=\emptyset$,
   $G^{E'}-E''=(G-E'')^{E'}$, we have
   $f(F')=v\Big(\big(G^{E(F)}-\Delta^{c}\big)^{\Delta}-\Delta\Big)=v\big((G^{E(F)}-\Delta^{c})/\Delta\big)$
   by \cref{def:Contraction}.
 \end{proof}

\begin{lemma}
   \label{lem:NOLoopQT}
   Let $G$ be a ribbon graph and $s$ a resolution of
   $G$ such that $H_{s}$ is a quasi-tree. Let $e$ be an edge of $G$, not necessarily in $E(H_{s})$. Let $s'$ be defined by
  \begin{align}
     s'=&
     \begin{cases}
       s&\text{on $E(G)\setminus\{e\}$},\\
       1-s&\text{on $\{e\}$}.
     \end{cases}
   \end{align}
   If $e$ is a non-orientable loop in $G^{E(H_{s})}$, then $H_{s'}$ is also a quasi-tree.
 \end{lemma}
\begin{proof}
   We are going to use \cref{prop:F'vsF} with $F=H_{s}$ and
   $F'=H_{s'}$. As $e\in H_{s}\iff e\notin H_{s'}$, $\Delta=\{e\}$. $F$ being a quasi-tree, $G^{E(F)}$ is a one-vertex
   ribbon graph and $G^{E(F)}-\Delta^{c}\fide H'$ consists of the unique
   vertex of $G^{E(F)}$ and the loop $e$. By \cref{prop:F'vsF} the number of faces of $F'$ equals the number of
   vertices of $H'/\Delta$. Proving that
   $F'$ is a quasi-tree amounts to proving that $H'/\{e\}$ is a
   one-vertex graph. By assumption $e$ is non-orientable in
   $G^{E(H_{s})}$. It is then non-orientable in $H'$. Thanks to the
   \cref{def:Contraction}, its contraction leads to a
   one-vertex ribbon graph.
 \end{proof}

\begin{lemma}
   \label{lem:Quasi-trees}
   Let $G$ be a ribbon graph and $s$ a resolution of
   $G$ such that $H_{s}$ is a quasi-tree. Let $e,e'$ be two
   edges of $G$, not necessarily in $E(H_{s})$. Let $s'$ be defined by
  \begin{align}
     s'=&
     \begin{cases}
       s&\text{on $E(G)\setminus\{e,e'\}$},\\
       1-s&\text{on $\{e,e'\}$}.
     \end{cases}
   \end{align}
   If $e$ and $e'$ link each other with respect to $H_{s}$ and at most one of them is
   a non-orientable loop in $G^{E(H_{s})}$, then $H_{s'}$ is also a quasi-tree.
 \end{lemma}

 \begin{proof}
   We distinguish between three cases: $1.$ $e,e'\in E(H_{s})$, $2.$
   neither $e$ nor $e'$ belongs to $E(H_{s})$ and $3.$ $e\in E(H_{s})$
   and $e'\notin E(H_{s})$ (or the converse). We are now going to use
   \cref{prop:F'vsF} with $F=H_{s}$ and $F'=H_{s'}$. In the
   three cases, $\Delta=\{e,e'\}$. $H_{s}=F$ being a quasi-tree,
   $G^{E(F)}$ is a one-vertex ribbon graph. Then $G^{E(F)}-\Delta^{c}$
   consists of the vertex of $G^{E(F)}$ and the two loops $e$ and
   $e'$. By assumption these link each other which means that
   they cross each other in $G^{E(F)}$.

   We have to consider two cases: $1.$ both $e$ and $e'$
   are orientable in $G^{E(H_{s})}$, $2.$ one of them is
   non-orientable, say $e$ and the other one $(e')$ is orientable.
   \begin{enumerate}
   \item The contraction of $e$ gives two vertices linked by a bridge
     $e'$. The contraction of $e'$ is a single vertex.
   \item The contraction of $e$ leads to a one-vertex ribbon graph
     with a single \emph{non-orientable} loop $e'$. The contraction of $e'$ leads to a single vertex and $f(F')=1$. 
   \end{enumerate}
\end{proof}

\subsection{Binary tree}
\label{sec:binary-tree}

\begin{defn}[Nugatory edges]\label{def:nugatory}
  Let $G$ be a ribbon graph and $\rho$ one of its partial
  resolutions. Let $e\in E(G)$ and $\rho_{0}^{e}$ (resp.\@
  $\rho_{1}^{e}$) be the partial resolution of $G$ obtained from
  $\rho$ by resolving $e$ to be $0$ (resp.\@ $1$). The edge $e$ is
  called \textbf{nugatory} if $[\rho_{0}^{e}]$ or
  $[\rho_{1}^{e}]$ does not contain any quasi-tree of $G$.
\end{defn}
For any connected ribbon graph $G$ and any total order on $E(G)$, we now describe the construction of the
binary tree $\cT(G)$. Each of its nodes is a partial resolution of
$G$. The construction essentially follows
\citep{Champanerkar2007aa}. Let the root of $\cT(G)$ be the totally
unresolved partial resolution of $G$: for all $e\in E(G)$,
$\rho(e)=*$. We resolve edges, in the reverse order (starting with the
highest edge), by changing $*$ to $0$ or $1$. If an edge is nugatory,
it is left unresolved and we proceed to the next edge. For a given
node $\rho$ in $\cT(G)$, if $e$ is not nugatory then the left child is
$\rho_{0}^{e}$ and the right child is $\rho_{1}^{e}$. We terminate
this process at a leaf when all subsequent edges are nugatory or all
edges have been resolved.\\

Let us now give an example of such a binary tree. We consider the ribbon graph of \cref{fig:CrossEdges} with $e_{1}\prec e_{2}\prec e_{3}$. The associated binary tree is represented in \cref{fig:BinaryTreeEx}. Each node of the tree is a partial resolution; for instance $*10$ corresponds to $\rho(e_{1})=*,\rho(e_{2})=1$ and $\rho(e_{3})=0$.
\begin{figure}[!htp]
  \centering
  \includegraphics[scale=.6]{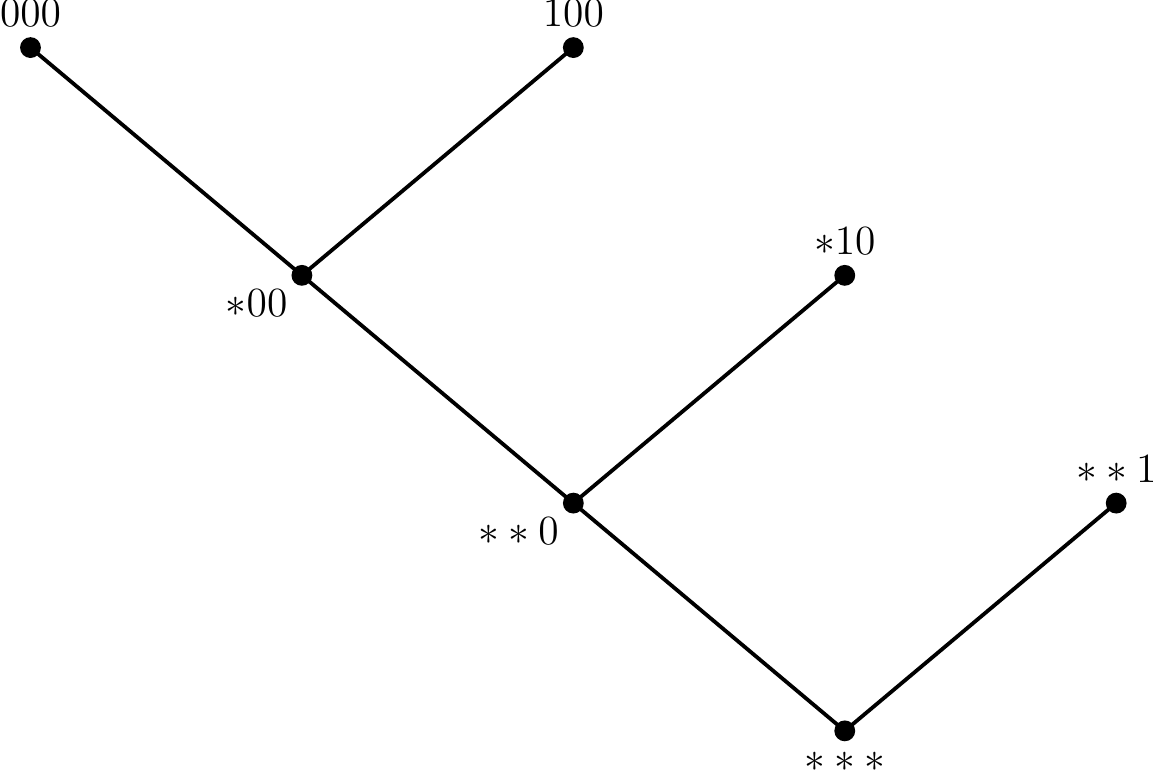}
  \caption{A binary tree of partial resolutions}
  \label{fig:BinaryTreeEx}
\end{figure}

By construction, each leaf $\rho$ of such a binary tree $\cT(G)$ is a
partial resolution of $G$ all the unresolved edges of which are
nugatory. Therefore there exists a unique resolution $s\in [\rho]$
such that $H_{s}$ is a quasi-tree. Indeed, let us consider a node of
the binary tree $\cT(G)$ i.e. a partial resolution $\sigma$ of $G$. Let
$e$ be the edge to be tested at this node. If $e$ is nugatory, either
$[\sigma^{e}_{0}]$ or $[\sigma^{e}_{1}]$ contains a quasi-tree. If $e$ is
not nugatory, they both contain a quasi-tree. Thus, by induction, for each leaf
$\rho$ of $\cT(G)$, $[\rho]$ contains at least one quasi-tree. Let us
assume that it contains more than one quasi-tree. This would mean that
there exists an unresolved edge $e$ in $\rho$ such that both
$[\rho_{0}^{e}]$ and $[\rho_{1}^{e}]$ contain a quasi-tree. But this
is in contradiction with the fact that all unresolved edges of a leaf are nugatory.

We let $\mathbf{Q_{\rho}}$ be the unique quasi-tree in $[\rho]$. For each spanning subribbon graph $H_{s},\,s\in [\rho]$ we define $Q_{H_{s}}$ to be $Q_{\rho}$.
\begin{lemma}
  \label{lem:LiveDead}
  Let $G$ be a connected ribbon graph. Let $\rho$ be a leaf of $\cT(G)$, and let $Q_{\rho}$ be the
  corresponding quasi-tree. If $e$ is unresolved in $\rho$ then $e$ is
  orientable in $G^{E(Q_{\rho})}$ and live with respect to $Q_{\rho}$. If $e$ is resolved in
  $\rho$, it is either dead with respect to $Q_{\rho}$ or non-orientable in $G^{E(Q_{\rho})}$
  and live.
\end{lemma}
\begin{proof}
  Let $e$ be an unresolved edge of a leaf $\rho$ of
  $\cT(G)$. If $e$ is non-orientable in $G^{E(Q_{\rho})}$
  then by \cref{lem:NOLoopQT} there exist two different resolutions
  in $[\rho]$ corresponding to quasi-trees. This
  contradicts the fact that $e$ is nugatory. As a conclusion,
  nugatory edges are orientable in $G^{E(Q_{\rho})}$.\\

  \noindent
  Let $e_{i}$ and $e_{j}$ be two unresolved edges in $\rho$, which are
  therefore nugatory and orientable in $G^{E(Q_{\rho})}$. If $e_{i}\lrtimes e_{j}$, by \cref{lem:Quasi-trees}, there exists two different resolutions in $[\rho]$
  corresponding to quasi-trees. This contradicts the fact that $e_{i}$ and
  $e_{j}$ are nugatory. Thus unresolved edges can only link resolved
  ones.\\

  \noindent
  Suppose $e_{i}$ is unresolved in $\rho$ and links a resolved edge
  $e_{j}$ with $j\prec i$. Let $s\in [\rho]$ be the resolution
  such that $H_{s}=Q_{\rho}$. The edge $e_{i}$ being unresolved in $\rho$, is
  orientable in $G^{E(Q_{\rho})}$, so we can apply \cref{lem:Quasi-trees}. Thus there exists another partial resolution
  $s'$ such that $f(H_{s'})=1$. $s'$ is obtained from $s$ by
  changing only $s(e_{i})$ and $s(e_{j})$.

  Now there exists a unique closest parent $\widetilde{\rho}$ of
  $\rho$ in $\cT(G)$ such that $e_{j}$ is a non-nugatory unresolved edge
  in $\widetilde{\rho}$. If $s\in[\widetilde{\rho}^{e_{j}}_{0}]$ (say)
  then $s'\in[\widetilde{\rho}^{e_{j}}_{1}]$. This implies that $e_{i}$
  is not nugatory in $\widetilde{\rho}$ and contradicts the
  assumption that $j\prec i$ because if that were the case and since edges are
  resolved in the reverse order, $e_{i}$ should be nugatory in
  $\widetilde{\rho}$. Thus if $e_{i}\lrtimes e_{j}$, $i\prec j$ and $e_{i}$
  is live.\\

  \noindent
  Finally, let $e_{i}$ be a resolved edge in $\rho$. If $e_{i}$ links
  an unresolved edge $e_{j}$ then by the previous argument $j\prec i$ and
  $e_{i}$ is dead. So let us assume that $e_{i}$ only links resolved
  edges $\{e_{j}\}_{j\in R},\,R\subset\{1,\dotsc,|E(G)|\}$. If there
  exists one $j\in R$ such that $j\prec i$, $e_{i}$ is dead. Suppose
  therefore that for all $j\in R,\,i\prec j$. There exists a unique closest parent $\widetilde{\rho}$ of
  $\rho$ in $\cT(G)$ such that $e_{i}$ is a non-nugatory unresolved edge
  in $\widetilde{\rho}$. Edges are resolved in reverse order, so
  the $e_{j}$'s, $j\in R$ are resolved in $\widetilde{\rho}$. Moreover
  both $[\widetilde{\rho}^{e_{i}}_{0}]$ and
  $[\widetilde{\rho}^{e_{i}}_{1}]$ contain a quasi-tree. If $e_{i}$ is
  orientable and does not link an unresolved edge, it is an orientable
  trivial loop in $G^{E(Q_{\rho})}-\{e_{j}\}_{j\in R}$. Suppose that
  $\rho\in [\widetilde{\rho}^{e_{i}}_{0}]$ (resp.\@
  $[\widetilde{\rho}^{e_{i}}_{1}]$). Then by proposition
  \cref{prop:F'vsF}, and since $\Delta$ and $\Delta^c$ being disjoint,
  we can change the order of contraction and deletion, for all $s\in
  [\widetilde{\rho}^{e_{i}}_{1}]$ (resp.\@
  $[\widetilde{\rho}^{e_{i}}_{0}]$), and $f(H_{s})=v(G^{E(Q_{\rho})}/\Delta-\Delta^{c})\ges 2$ with $e_{i}\in\Delta$ and for all $j\in R$, $e_{j}\notin\Delta$. Thus either $[\widetilde{\rho}^{e_{i}}_{0}]$ or $[\widetilde{\rho}^{e_{i}}_{1}]$
  does not contain any quasi-tree which contradicts the fact that
  $e_{i}$ is resolved. Therefore $e_{i}$ links an unresolved edge and
  is dead.\\
Note finally that if $R=\emptyset$ i.e. if $e_{i}$ does not link any
edge, exactly the same reasoning applies as well. Namely, if
$e_{i}\in[\widetilde{\rho}_{0}^{e_{i}}]$
(resp. $e_{i}\in[\widetilde{\rho}_{1}^{e_{i}}]$), $[\widetilde{\rho}_{1}^{e_{i}}]$
(resp. $[\widetilde{\rho}_{0}^{e_{i}}]$) does not contain any
quasi-tree. This contradicts the fact that $e_{i}$ is resolved in
$\rho$ and proves that $e_{i}$ links an unresolved edge.
\end{proof}
\begin{rem}
  Concerning the last part of the preceding proof, if $e_{i}$ is non-orientable and only links higher-ordered edges, it does not need to link an unresolved edge to ensure that both $[\widetilde{\rho}^{e_{i}}_{0}]$ and
    $[\widetilde{\rho}^{e_{i}}_{1}]$ contain a quasi-tree. Thus non-orientable (resolved) edges may be live. For example, in the leaf $100$ of the binary tree in \cref{fig:BinaryTreeEx} (which corresponds to the graph of \cref{fig:CrossEdges}), the edge $e_{1}$ is non-orientable, resolved and live.
\end{rem}
To sum up this section, we have proven the following
\begin{cor}\label{cor:SpannSubgBij}
Let $G$ be a connected ribbon graph and $\cS_{G}$ its set of spanning subribbon graphs. Given a total order on $E(G)$, $\cS_{G}$ is in one-to-one correspondence with $\bigcup_{Q\in\cQ_{G}}\cI_{\text{o}}(Q)\times\cE_{\text{o}}(Q)$. Namely to each spanning subribbon graph $F$ there corresponds a unique quasi-tree $Q_{F}$. Then, there exists $S\subseteq \cI_{\text{o}}(Q_{F})\cup\cE_{\text{o}}(Q_{F})$ such that $E(F)=\cD(Q_{F})\cup\cI_{\text{n}}(Q_{F})\cup S$.
\end{cor}
\newpage
\section{Non-orientable quasi-tree expansions}
\label{sec:non-orientable-quasi}

\subsection{The (signed) \BRp}
\label{sec:brp}

This section is devoted to the statement and proof of our main theorem, namely a quasi-tree expansion of the signed \BRp{} of not necessarily orientable ribbon graphs. For any subribbon graph $F$ of $G$, we let $t(F)$ be $0$ if $F$ is orientable and $1$ otherwise. Recall that for any ribbon graph $G$, the (unsigned) \BRp{} is defined by \citep{Bollobas2002aa}
\begin{align}
  \label{eq:BRpDef}
  R(G;x,y,z,w)=&\sum_{F\subseteq G}(x-1)^{r(G)-r(F)}y^{n(F)}z^{(k-f+n)(F)}w^{t(F)}
\end{align}
considered as an element of the quotient of $\Z[x,y,z,w]$ by the ideal generated by $w^{2}-w$. 

S.~Chmutov and I.~Pak introduced an extension of the \BRp{} at $w=1$ \citep{Chmutov2007ab}. It is a three-variable \textbf{polynomial} $R_{s}$ defined on \emph{signed} ribbon graphs. Recall that a graph is said to be signed if to each of its edges, an element of $\{+,-\}$ is assigned.

For any signed ribbon graph $G$, let $E_{+}(G)$ (resp.\@ $E_{-}(G)$) be the set of positive (resp.\@ negative) edges of $G$, and let $e_{\pm}(G)$ be their respective cardinalities. For any spanning subribbon graph $F$ of $G$, let $\bar{F}$ denote the spanning subribbon graph of $G$ with edge-set $E(F)^{c}$. Let us finally define $s(F)\defi\frac 12(e_{-}(F)-e_{-}(\bar{F}))$. The signed \BRp{} is
\begin{align}
  \label{eq:SignedBRpDef}
  R_{s}(G;x+1,y,z)=&\sum_{F\subseteq G}x^{k(F)-k(G)+s(F)}y^{n(F)-s(F)}z^{(k-f+n)(F)}.
\end{align}
If all the edges of $G$ are positive, $R_{s}(G;x,y,z)=R(\widetilde{G};x,y,z,1)$ where $\widetilde{G}$ is the underlying unsigned ribbon graph in $G$.\\

Before stating our main theorem, we need to recall the definition of the rank polynomial of C.~Godsil and G.~Royle \citep{Godsil2001aa}. It is a four-variable polynomial defined on matroids. Nevertheless, restricting ourselves to graphic matroids, we can easily deduce a version of this polynomial for graphs.
\begin{defn}[The Rank polynomial \citep{Godsil2001aa}]\label{def:RankPDef}
  Let $G$ be a graph (not a ribbon graph). The rank polynomial is defined as follows:
  \begin{align}
    \label{eq:RankPDef}
    Ra(G;\alpha,\beta,\gamma,\delta)=&\sum_{F\subseteq G}\alpha^{e_{+}(\Fb)+e_{-}(F)}\beta^{e_{+}(F)+e_{-}(\Fb)}\gamma^{k(F)-k(G)}\delta^{n(F)}
  \end{align}
where the sum runs over the spanning subgraphs of $G$.
\end{defn}
Note that the rank polynomial is homogeneous in $\alpha,\beta$: the sum of the exponents of $\alpha$ and $\beta$ is constant and equals $e(G)$. Thus we have
\begin{align}
  \label{eq:HomogRkP}
  Ra(G;\alpha,\beta,\gamma,\delta)=&\alpha^{e(G)}Ra(G;1,\beta/\alpha,\gamma,\delta).
\end{align}
The rank polynomial is a generalization of the \Tp{}:
\begin{align}
  T(G;x,y)\defi\sum_{F\subseteq G}(x-1)^{k(F)-k(G)}(y-1)^{n(F)}=Ra(G;1,1,x-1,y-1).
\end{align}\\

The signed and unsigned \BRp s are multiplicative on disjoint unions of ribbon graphs, so we can restrict ourselves to connected ribbon graphs, without loss of generality.
\begin{defn}\label{def:GQ}
  Let $G$ be a connected ribbon graph. For any total order on $E(G)$
  and any quasi-tree $Q\in\cQ_{G}$, let $\mathbf{G_{Q}}$ be the graph
  (not the ribbon graph) whose vertices are the components of
  $F_{\cD(Q)\cup\cI_{\text{n}}(Q)}$ and whose edges are the
  internally live orientable edges (namely the elements of
  $\cI_{\text{o}}(Q)$). In other words, consider the graph $\widetilde
  G$ underlying $G$. There is obviously a bijection $f$ between $E(G)$
  and $E(\widetilde G)$. Then $G_{Q}\defi \widetilde
  G/f(\cD(Q)\cup\cI_{\text{n}}(Q))$ (remember that, in a graph, the contraction of
  a loop consists in its deletion).
\end{defn}
\begin{thm}[Quasi-tree expansion]
  \label{thm:QTSBRP}
  Let $G$ be a connected signed ribbon graph. For any total order on $E(G)$, the signed \BRp{} is given by
  \begin{align}\label{eq:SBRpExp}
    R_{s}(G;x+1,y,z)=&(x^{-1/2}y^{1/2})^{\,e_{-}(G)}\sum_{Q\in\cQ_{G}}x^{e_{-}(\cD(Q)\cup\cI_{\text{n}}(Q))}y^{n(F_{\cD(Q)\cup\cI_{\text{n}}(Q)})-e_{-}(\cD(Q)\cup\cI_{\text{n}}(Q))}\nonumber\\ 
&z^{(k-f+n)(F_{\cD(Q)\cup\cI_{\text{n}}(Q)})} (1+x)^{e_{-}(\cE_{\text{o}}(Q))}(1+y)^{e_{+}(\cE_{\text{o}}(Q))}\nonumber\\
&(x^{1/2}y^{-1/2})^{r(G_{Q})+e_{-}(\cI_{\text{o}}(Q))}Ra(G_{Q};1,x^{-1/2}y^{1/2},x^{1/2}y^{1/2},x^{1/2}y^{1/2}z^{2})
  \end{align}
where, for all $E'\subseteq E(G)$, $E_{\pm}(E')\defi E_{\pm}(G)\cap
E'$, and $e_{\pm}(E')\defi |E_{\pm}(G)\cap E'|$.
\end{thm}
\begin{cor}
  \label{cor:QTBRP}
  Let $G$ be a connected ribbon graph. For any total order on $E(G)$, the \BRp{} at $w=1$ is given by
  \begin{align}\label{eq:BRpExp}
    R(G;x,y,z,1)=&\sum_{Q\in\cQ_{G}}y^{n(F_{\cD(Q)\cup
        \cI_{\text{n}}(Q)})}z^{(k-f+n)(F_{\cD(Q)\cup
        \cI_{\text{n}}(Q)})}(1+y)^{|\cE_{\text{o}}(Q)|}T(G_{Q};x,yz^{2}+1)\nonumber
  \end{align}
  where $T(G_{Q})$ is the \Tp{} of $G_{Q}$.
\end{cor}
Before proving \cref{thm:QTSBRP}, let us comment on the fact that,
in \cref{cor:QTBRP}, we get a quasi-tree expansion only at
$w=1$. To extend our expansion to the full \BRp{} (namely for any
$w$), we would need in particular to relate the orientability of any subgraph to the
orientability of $F_{\cD\cup\cI_{\text{n}}}$. This has been done in \citep{Dewey2007aa}.\\

The proof of \cref{thm:QTSBRP} relies on the following lemma:
\begin{lemma}
  \label{lem:kf}
  Let $G$ be a connected ribbon graph. Let $Q\in\cQ_{G}$ be a quasi-tree in $G$. Given a total order on $E(G)$, and for any $S=S_{1}\cup S_{2}$ with $S_{1}\subset\cI_{\text{o}}(Q)$ and $S_{2}\subset\cE_{\text{o}}(Q)$, we have
  \begin{itemize}
  \item $k(F_{\cD(Q)\cup \cI_{\text{n}}(Q)\cup S})=k(F_{\cD(Q)\cup \cI_{\text{n}}(Q)\cup S_{1}})=k(W)$, where $W$ is the spanning subgraph of $G_{Q}$, the edge-set of which is $S_{1}$,
  \item $f(F_{\cD(Q)\cup \cI_{\text{n}}(Q)\cup S})=f(F_{\cD(Q)\cup \cI_{\text{n}}(Q)})-|S_{1}|+|S_{2}|$.
  \end{itemize}
\end{lemma}
\begin{proof}
  The edges in $S$ being orientable, the proof follows the one given in \citep{Champanerkar2007aa}. But we reformulate it in terms of S.~Chmutov's duality.

  Let $e\in S_{2}$. We want to prove that $k(F_{\cD(Q)\cup
    \cI_{\text{n}}(Q)\cup S})=k(F_{\cD(Q)\cup \cI_{\text{n}}(Q)\cup
    S\setminus\{e\}})$ that is to say that $e$ intersects only one
  component of $F_{\cD(Q)\cup \cI_{\text{n}}(Q)\cup
    S\setminus\{e\}}$. Clearly if $e$ intersects only one component of
  $F_{\cD(Q)\cup\cI_{\text{n}}(Q)\cup S_{1}}$, it does so a fortiori in $F_{\cD(Q)\cup \cI_{\text{n}}(Q)\cup S\setminus\{e\}}$. Then it is enough to prove that $k(F_{\cD(Q)\cup \cI_{\text{n}}(Q)\cup S_{1}\cup\{e\}})=k(F_{\cD(Q)\cup \cI_{\text{n}}(Q)\cup S_{1}})$. Actually we are going to prove an even stronger statement, namely that $e$ only intersects one boundary component of $F_{\cD(Q)\cup \cI_{\text{n}}(Q)\cup S_{1}}$. This would obviously imply the desired result.

 The boundary components of a ribbon graph are the vertices of its natural dual. We will therefore prove that $e$ is a loop in $(F_{\cD(Q)\cup \cI_{\text{n}}(Q)\cup S_{1}\cup\{e\}})^{\cD(Q)\cup \cI_{\text{n}}(Q)\cup S_{1}}$.
  \begin{align}
    (F_{\cD(Q)\cup \cI_{\text{n}}(Q)\cup S_{1}\cup\{e\}})^{\cD(Q)\cup
      \cI_{\text{n}}(Q)\cup
      S_{1}}=&G^{E(Q)}/\bar{S_{1}}-\big((\cE_{\text{o}}(Q)\setminus\{e\})\cup\cE_{\text
    n}(Q)\cup\kD(Q)\big)\\
  \ifed&G^{E(Q)}/\bar{S_{1}}-A
  \end{align}
  with $\bar{S_{1}}\defi\cI_{\text{o}}(Q)\setminus S_{1}$, $\kD(Q)$ the set of externally dead edges and where we
  used $E(Q)=\cD(Q)\cup\cI_{\text{n}}(Q)\cup \cI_{\text{o}}(Q)$ and
  \cref{def:Contraction}. $Q$ being a quasi-tree,
  $G^{E(Q)}-A$ is a one-vertex ribbon
  graph with edges in $\cD(Q)\cup \cI_{\text{n}}(Q)\cup
  \cI_{\text{o}}(Q)\cup\{e\}$. The edges in $\bar{S_{1}}\cup\{e\}$ are
  all unresolved in the partial resolution $\rho$ of $\cT(G)$ such
  that $Q_{\rho}=Q$. Therefore they do not cross each other in $G^{E(Q)}$; see the proof of \cref{lem:LiveDead}. As a consequence the edge $e$ is still a loop in $G^{E(Q)}/\bar{S_{1}}$ and the first equality of the first item follows.

  The proof that $k(F_{\cD(Q)\cup \cI_{\text{n}}(Q)\cup S_{1}})=k(W)$ is obvious from the \cref{def:GQ} of $G_{Q}$.

  Let us now prove the second statement of the lemma:
  \begin{align}
    f(F_{\cD(Q)\cup \cI_{\text{n}}(Q)\cup S})=&v(F_{\cD(Q)\cup \cI_{\text{n}}(Q)\cup S}^{\star})=v(G^{\cD(Q)\cup \cI_{\text{n}}(Q)\cup S})=v\big(G^{E(Q)}/(\bar{S_{1}}\cup S_{2})\big)\\
    f(F_{\cD(Q)\cup \cI_{\text{n}}(Q)})=&v\big(G^{E(Q)}/\cI_{\text{o}}(Q)\big)
    \intertext{But the edges in
      $\cI_{\text{o}}(Q)\cup\cE_{\text{o}}(Q)$ do not cross each other
      in $G^{E(Q)}$ (see the proof of \cref{lem:LiveDead}). Thus, given the \cref{def:Contraction} of the contraction of a loop, we have}
    f(F_{\cD(Q)\cup \cI_{\text{n}}(Q)\cup S})=&v(G^{E(Q)})+|\bar{S_{1}}|+|S_{2}|,\\
    f(F_{\cD(Q)\cup \cI_{\text{n}}(Q)})=&v(G^{E(Q)})+|S_{1}|+|\bar{S_{1}}|
  \end{align}
which implies $f(F_{\cD(Q)\cup \cI_{\text{n}}(Q)\cup S})=f(F_{\cD(Q)\cup \cI_{\text{n}}(Q)})-|S_{1}|+|S_{2}|$.
\end{proof}

\begin{cor}\label{cor:nAndg}
  Let $G$ be a connected ribbon graph. Let $Q\in\cQ_{G}$ be a quasi-tree in $G$. Given a total order on $E(G)$ and for any $S=S_{1}\cup S_{2}$ with $S_{1}\subset\cI_{\text{o}}(Q)$ and $S_{2}\subset\cE_{\text{o}}(Q)$, we have
  \begin{itemize}
  \item $n(F_{\cD(Q)\cup\cI_{\text{n}}(Q)\cup S})=n(F_{\cD(Q)\cup\cI_{\text{n}}(Q)})+n(W)+|S_{2}|$,
  \item $(k-f+n)(F_{\cD(Q)\cup\cI_{\text{n}}(Q)\cup S})=(k-f+n)(F_{\cD(Q)\cup\cI_{\text{n}}(Q)})+2n(W)$.
  \end{itemize}
\end{cor}
\begin{proof}
  Using now \cref{lem:kf}, we have
   \begin{align}
          n(F_{\cD(Q)\cup\cI_{\text{n}}(Q)\cup S})=&(e-v+k)(F_{\cD(Q)\cup\cI_{\text{n}}(Q)\cup S})\\
     =&e(F_{\cD(Q)\cup\cI_{\text{n}}(Q)\cup S_{1}})+|S_{2}|-v(G)+k(F_{\cD(Q)\cup\cI_{\text{n}}(Q)\cup S_{1}})\\
     =&n(F_{\cD(Q)\cup\cI_{\text{n}}(Q)\cup S_{1}})+|S_{2}|,\label{eq:n1}\\
     \nonumber\\
     n(W)=&e(W)-v(W)+k(W)\\
     =&|S_{1}|-k(F_{\cD(Q)\cup\cI_{\text{n}}(Q)})+k(F_{\cD(Q)\cup\cI_{\text{n}}(Q)\cup S_{1}}),\\
     \nonumber\\
     n(F_{\cD(Q)\cup\cI_{\text{n}}(Q)\cup S_{1}})=&e(F_{\cD(Q)\cup\cI_{\text{n}}(Q)})+|S_{1}|-v(G)+k(F_{\cD(Q)\cup\cI_{\text{n}}(Q)\cup S_{1}})\\
     =&e(F_{\cD(Q)\cup\cI_{\text{n}}(Q)})-v(G)+k(F_{\cD(Q)\cup\cI_{\text{n}}(Q)})\nonumber\\
     &+|S_{1}|-k(F_{\cD(Q)\cup\cI_{\text{n}}(Q)})+k(F_{\cD(Q)\cup\cI_{\text{n}}(Q)\cup S_{1}})\\
     =&n(F_{\cD(Q)\cup\cI_{\text{n}}(Q)})+n(W)\label{eq:n2}
     \intertext{Equations (\cref{eq:n1}) and (\cref{eq:n2}) imply $n(F_{\cD(Q)\cup\cI_{\text{n}}(Q)\cup S})=n(F_{\cD(Q)\cup\cI_{\text{n}}(Q)})+n(W)+|S_{2}|$.}
     (k-f+n)(F_{\cD(Q)\cup\cI_{\text{n}}(Q)\cup S})=&k(F_{\cD(Q)\cup\cI_{\text{n}}(Q)\cup S_{1}})-f(F_{\cD(Q)\cup\cI_{\text{n}}(Q)})+|S_{1}|-|S_{2}|\nonumber\\
     &+n(F_{\cD(Q)\cup\cI_{\text{n}}(Q)})+n(W)+|S_{2}|\\
     =&(k-f+n)(F_{\cD(Q)\cup\cI_{\text{n}}(Q)})+k(F_{\cD(Q)\cup\cI_{\text{n}}(Q)\cup S_{1}})\nonumber\\
     &-k(F_{\cD(Q)\cup\cI_{\text{n}}(Q)})+|S_{1}|+n(W)\\
     =&(k-f+n)(F_{\cD(Q)\cup\cI_{\text{n}}(Q)})+2n(W)
    \end{align}
    which proves \cref{cor:nAndg}.
\end{proof}

\begin{proof}[of \cref{thm:QTSBRP}]
  Thanks to \cref{cor:SpannSubgBij}, the signed \BRp{} can be written as follows
     \begin{align}
R_{s}(G;x+1,y,z)=&(x^{-1/2}y^{1/2})^{e_{-}(G)}\sum_{Q\in\cQ_{G}}\sum_{S_{1}\subset\cI_{\text{o}}(Q)}\sum_{S_{2}\subset\cE_{\text{o}}(Q)}x^{k(F_{\cD(Q)\cup \cI_{\text{n}}(Q)\cup S})-k(G)+e_{-}(\cD(Q)\cup \cI_{\text{n}}(Q)\cup S)}\nonumber\\
        &\qquad y^{n(F_{\cD(Q)\cup \cI_{\text{n}}(Q)\cup S})-e_{-}(\cD(Q)\cup \cI_{\text{n}}(Q) \cup S)}z^{(k-f+n)(F_{\cD(Q)\cup \cI_{\text{n}}(Q)\cup S})}
      \end{align}
  where $S=S_{1}\cup S_{2}$. Using now \cref{lem:kf} and \cref{cor:nAndg}, we have
   \begin{align}
     R_{s}(G;x+1,y,z)=&(x^{-1/2}y^{1/2})^{e_{-}(G)}\sum_{Q\in\cQ_{G}}x^{e_{-}(\cD(Q)\cup\cI_{\text{n}}(Q))}y^{n(F_{\cD(Q)\cup\cI_{\text{n}}(Q)})-e_{-}(\cD(Q)\cup\cI_{\text{n}}(Q))}\nonumber\\
     &\hspace{-2.7cm}z^{(k-f+n)(F_{\cD(Q)\cup\cI_{\text{n}}(Q)})}\sum_{\mathclap{S_{2}\subseteq\cE_{\text
         o}(Q)}}x^{e_{-}(S_{2})}y^{e_{+}(S_{2})}\sum_{W\subseteq G_{Q}}x^{k(W)-k(G_{Q})+e_{-}(W)}(yz^{2})^{n(W)}y^{-e_{-}(W)}\\
     =&(x^{-1/2}y^{1/2})^{e_{-}(G)}\sum_{Q\in\cQ_{G}}x^{e_{-}(\cD(Q)\cup\cI_{\text{n}}(Q))}y^{n(F_{\cD(Q)\cup\cI_{\text{n}}(Q)})-e_{-}(\cD(Q)\cup\cI_{\text{n}}(Q))}\nonumber\\
     &z^{(k-f+n)(F_{\cD(Q)\cup\cI_{\text{n}}(Q)})} (1+x)^{e_{-}(\cE_{\text{o}}(Q))}(1+y)^{e_{+}(\cE_{\text{o}}(Q))}\nonumber\\
     &\sum_{W\subseteq G_{Q}}x^{k(W)-k(G_{Q})+e_{-}(W)}(yz^{2})^{n(W)}y^{-e_{-}(W)}
   \end{align}
   where we used $k(G)=k(G_{Q})=1$. To conclude, it remains to prove that
   \begin{align}
     &\sum_{W\subseteq G_{Q}}x^{k(W)-k(G_{Q})+e_{-}(W)}(yz^{2})^{n(W)}y^{-e_{-}(W)}\nonumber\\
     =&(x^{1/2}y^{-1/2})^{r(G_{Q})+e_{-}(\cI_{\text{o}}(Q))}Ra(G_{Q};1,x^{-1/2}y^{1/2},x^{1/2}y^{1/2},x^{1/2}y^{1/2}z^{2})
   \end{align}
   which is easily checked from \cref{def:RankPDef} of the rank polynomial.
\end{proof}
\cref{cor:QTBRP} is a direct consequence of theorem
\cref{thm:QTSBRP}. It is indeed easily verified that, if $G$ is a
signed ribbon graph with only positive edges, the right hand side of (\cref{eq:SBRpExp}) reduces to the desired expression of \cref{cor:QTBRP}.

\subsection{The multivariate \BRp}
\label{sec:multivariate-brp}

Multivariate versions of (ribbon) graph polynomials consist in attaching a different indeterminate to each edge. The multivariate \BRp{} is defined as follows \citep{Moffatt2008ab}: let $G$ be a ribbon graph,
\begin{align}
  \label{eq:MultBRDef}
  Z(G;q,\mathbf{\beta},c)\defi&\sum_{F\subseteq G}q^{k(F)}\Big(\prod_{e\in E(F)}\beta_{e}\Big)c^{f(F)}
\end{align}
where $\mathbf{\beta}=\{\beta_{e}\tqs e\in E(G)\}$. Let $G$ be a graph; the multivariate Tutte polynomial is defined as \citep{Traldi1989aa}
\begin{align}
  \label{eq:MultiTutteDef}
  Z_{T}(G;q,\mathbf{\beta})\defi&\sum_{F\subseteq G}q^{k(F)}\Big(\prod_{e\in E(F)}\beta_{e}\Big).
\end{align}
\begin{lemma}
  \label{lem:MultiBR-QT}
  Let $G$ be a connected ribbon graph. For any total order on $E(G)$, the multivariate \BRp{} $Z$ is given by
  \begin{align}
    Z(G;q,\mathbf{\beta},c)=&\sum_{Q\in\cQ_{G}}\Big(\prod_{e\in\cD(Q)\cup\cI_{\text{n}}(Q)}\beta_{e}\Big)c^{f(F_{\cD(Q)\cup\cI_{\text{n}}(Q)})}\Big(\prod_{e\in\cE_{\text{o}}(Q)}(1+c\beta_{e})\Big)Z_{T}(G_{Q};q,\mathbf{\beta}/c).\nonumber
  \end{align}
\end{lemma}
\begin{proof}
    Thanks to \cref{cor:SpannSubgBij}, the multivariate \BRp{} can be written as follows
    \begin{multline}
      Z(G;q,\mathbf{\beta},c)=\sum_{Q\in\cQ_{G}}\sum_{S_{1}\subset\cI_{\text{o}}(Q)}
      \sum_{S_{2}\subset\cE_{\text{o}}(Q)}q^{k(F_{\cD(Q)\cup \cI_{\text{n}}(Q)\cup S})}
      \Big(\prod_{e\in \cD(Q)\cup \cI_{\text{n}}(Q)\cup S}\beta_{e}\Big)c^{f(F_{\cD(Q)\cup \cI_{\text{n}}(Q)\cup S})}
    \end{multline}
   where $S=S_{1}\cup S_{2}$. Using now \cref{lem:kf}, we have
   \begin{align}
     Z(G;q,\mathbf{\beta},c)=&\sum_{Q\in\cQ_{G}}\Big(\prod_{e\in \cD(Q)\cup \cI_{\text{n}}(Q)\cup S}\beta_{e}\Big)c^{f(F_{\cD(Q)\cup \cI_{\text{n}}(Q)})}\sum_{S_{2}}\Big(\prod_{e\in S_{2}}c\beta_{e}\Big)\nonumber\\
     &\times\sum_{S_{1}}q^{k(F_{\cD(Q)\cup \cI_{\text{n}}(Q)\cup S_{1}})}\Big(\prod_{e\in S_{1}}\beta_{e}/c\Big)\\
     =&\sum_{Q\in\cQ_{G}}\Big(\prod_{e\in \cD(Q)\cup \cI_{\text{n}}(Q)\cup S}\beta_{e}\Big)c^{f(F_{\cD(Q)\cup \cI_{\text{n}}(Q)})}\Big(\prod_{e\in\cE_{\text{o}}(Q)}(1+c\beta_{e})\Big)\nonumber\\
     &\times\sum_{W\subseteq G_{Q}}q^{k(W)}\Big(\prod_{e\in E(W)}\beta_{e}/c\Big)
   \end{align}
and the lemma follows.
\end{proof}

In \citep{Vignestourneret2008aa} a multivariate extension of this signed polynomial has been defined and studied. Its invariance under the partial duality has also been proven in \citep{Vignestourneret2008aa}.
The multivariate signed \BRp{} is defined as follows:
\begin{align}
  \label{eq:MutliSignBRDef}
  Z_{s}(G;q,\mathbf{\alpha},c)\defi&\sum_{F\subseteq G}q^{k(F)+s(F)}\Big(\prod_{\substack{e\in E_{+}(F)\\\cup E_{-}(\bar{F})}}\alpha_{e}\Big)c^{f(F)}.
\end{align}
It is a multivariate generalization of $R_{s}$. Indeed if for any $e\in E(G),\,\alpha_{e}=yz$ and if we let $\mathbf{yz}$ be the corresponding set, we have
\begin{align}
  R_{s}(G;x+1,y,z)=&x^{-k(G)}(yz)^{-v(G)}Z_{s}(G;xyz^{2},\mathbf{yz},z^{-1}).\label{eq:GenerRs}
\end{align}
The multivariate polynomial $Z_{s}$ is actually related to the (unsigned) multivariate \BRp{} by
\begin{align}
  Z_{s}(G;q,\balpha,c)=&\Big(\prod_{e\in E_{-}(G)}q^{-1/2}\alpha_{e}\Big)Z(G;q,\mathbf{\beta},c)\label{eq:ZZhat}\\
  \text{with }\beta_{e}=&%
  \begin{cases}
    \alpha_{e}&\text{if $e$ is positive,}\\
    q\alpha_{e}^{-1}&\text{if $e$ is negative.}
  \end{cases}\label{eq:Beta}
\end{align}
It is then an easy exercise to get a quasi-tree expansion for the
signed multivariate \BRp{} from \cref{lem:MultiBR-QT}.

\section{Duality properties}
\label{sec:duality-properties}

In this section, we first recover the duality property of the \BRp{}, namely its invariance at $q=xyz^{2}=1$ \citep{Chmutov2007aa,Vignestourneret2008aa}, but via its quasi-tree expansion. As a consequence, we get another expression for the \BRp{} at $q=1$.\\

In \citep{Chmutov2007aa,Vignestourneret2008aa}, it has been proven
that, for any signed ribbon graph $G$ and any subset $E'\subseteq E(G)$ of edges,
\begin{subequations}
  \begin{align}
    Z_{s}(G;1,\balpha,c)=&Z_{s}(G^{E'};1,\balpha,c),\label{eq:DualityTransfoZ}\\
    \text{where
    }Z_{s}(G;xyz^{2},\mathbf{yz},z^{-1})\defi&x^{k(G)}(yz)^{v(G)}R_{s}(G;x+1,y,z).\label{eq:CorrespZRs}
  \end{align}
\end{subequations}
To prove equation (\cref{eq:DualityTransfoZ}), S.~Chmutov first exhibited a bijection between the subribbon graphs of $G$ and those of $G^{E'}$. Let us write $\cS_{G}$ for the set of spanning subribbon graphs of $G$. The bijection is the following map:
\begin{equation}
\begin{split}
  \varphi:\cS_{G}\to&\cS_{G^{E'}}\\
  F\mapsto&F'\text{ s.t. }E(F')=E'\Delta E(F),
\end{split}\label{eq:Bijection}
\end{equation}
where $\Delta$ stands for the symmetric difference. Then, defining
\begin{align}
  \label{eq:Monomials}
  Z_{s}(G;q,\balpha,c)\fide&\sum_{F\in\cS_{G}}M_{G}(F;q,\balpha,c),
\end{align}
he proved that $M_{G}(F;1,\alpha,c)=M_{G^{E'}}(\varphi(F);1,\alpha,c)$.\\

The quasi-tree expansion (\cref{eq:SBRpExp}) (or \cref{lem:MultiBR-QT}) is a way to factorize some of the monomials $M_{G}(F)$, naturally associated with a single quasi-tree $Q$ of $G$. Defining
\begin{align}
  \label{eq:MonomialsQT}
  Z_{s}(G;q,\balpha,c)\fide&\sum_{Q\in\cQ_{G}}N_{G}(Q;q,\balpha,c),
\end{align}
each monomial $N_{G}(Q)$ is the sum of several $M_{G}(F)$s. In the following, we prove that the bijection (\cref{eq:Bijection}) also preserves the $N_{G}(Q)$s:
\begin{lemma}
  \label{lem:NQ}
  For any signed ribbon graph $G$ and any subset $E'\subseteq E(G)$,\\
  $N_{G}(Q;1,\balpha,c)=N_{G^{E'}}(\varphi(Q);1,\balpha,c)$.
\end{lemma}
In the following, if $P$ is a rational function in one variable, and for all $A\subseteq E(G)$, we abbreviate
$\prod_{e\in A}P(\alpha_{e})$ as $\mathbf{P(\alpha)^{A}}$.
\begin{proof}
  First, note that, from \cref{lem:MultiBR-QT} and equation (\cref{eq:ZZhat}),
  \begin{align}
    \label{eq:NQ-Expr}
    N_{G}(Q;q,\balpha,c)\ifed&(\alpha/\sqrt q)^{E_{-}(G)}\alpha^{E_{+}(\cD\cup\cI_{\text{n}}(Q))}(q/\alpha)^{E_{-}(\cD\cup\cI_{\text{n}}(Q))}c^{f(F_{\cD\cup\cI_{\text{n}}(Q)})}\nonumber\\
    &\quad (1+\alpha c)^{E_{+}(\cE_{\text{o}})}(1+qc/\alpha)^{E_{-}(\cE_{\text{o}})}\,Z_{R}(G_{Q};q,\bbeta/c),\\
    \intertext{where $\bbeta$ is given by equation (\cref{eq:Beta}), so that}
    Z_{R}(G_{Q};q,\bbeta/c)\defi&\sum_{F\subseteq G_{Q}}q^{k(F)+e_{-}(F)}(\alpha/c)^{E_{+}(F)}(\alpha c)^{-E_{-}(F)}.
  \end{align}
For $q=1$, we can explicitly perform the summation over the spanning subgraphs of $G_{Q}$ to get
\begin{align}
  \label{eq:NQq1}
  N_{G}(Q;1,\balpha,c)=&\alpha^{E_{-}(G)+E_{+}(\cD\cup\cI_{\text{n}}(Q))-E_{-}(\cD\cup\cI_{\text{n}}(Q))-E_{-}(\cI_{\text{o}}\cup\cE_{\text{o}}(Q))}c^{f(F_{\cD\cup\cI_{\text{n}}(Q)})-|\cI_{\text{o}}(Q)|}\nonumber\\
  &\quad (1+\alpha c)^{E_{+}(\cE_{\text{o}}(Q))+E_{-}(\cI_{\text{o}}(Q))}(\alpha+c)^{E_{-}(\cE_{\text{o}}(Q))+E_{+}(\cI_{\text{o}}(Q))}.
\end{align}

To prove the lemma, let us first prove that the bijection $\varphi$ conserves the number of faces:
\begin{align}
  f(F')=&v(F'^{\star})=v((G^{E'})^{E(F')})=v((G^{E'})^{E'\Delta E(F)})=v(G^{E(F)})=f(F).
\end{align}
This implies that, if $Q$ is a quasi-tree of $G$, then $\varphi(Q)$ is
a quasi-tree of $G^{E'}$. Moreover, as defined in section
\cref{sec:activities-wrt-quasi}, an edge $e\in E(G)$ is live
(resp. orientable) with
respect to $Q$ if it does not cross any lower-ordered edge (resp. if
it is an orientable loop) in
$G^{E(Q)}$. Then, an edge $e\in E(G^{E'})$ is live (resp. orientable) with respect to
$\varphi(Q)$ if it does not cross any lower-ordered edge (resp. if it
is an orientable loop) in
$(G^{E'})^{E(\varphi(Q))}=(G^{E'})^{E'\Delta E(Q)}=G^{E(Q)}$. Thus,
the sets of orientable (resp.\@ non-orientable) live (and dead) edges
with respect to $Q$ in $G$ and with respect to $\varphi(Q)$ in
$G^{E'}$ are the same. Nevertheless, as $E(Q)$ and
$E(\varphi(Q))=E'\Delta E(Q)$ are
different, some internal edges with respect to $Q$ may be external
with respect to $\varphi(Q)$, and vice versa. For example, the (internal) edges of
$G^{E'}$ with respect to $F_{\varphi(Q)}$ (i.e. $\varphi(Q)$) contain
both internal edges (the ones in
$E(Q)\setminus E'$) and external edges (the ones in $E'\setminus
E(Q)$) of $G$ with respect to $Q$. In other words, having
\begin{subequations}
  \begin{align}
    E(G)=&(\cD\cup\cI_{\text{o}}\cup\cI_{\text{n}})(Q)\cup(\kD\cup\cE_{\text{o}}\cup\cE_{\text{n}})(Q),\\
    E(Q)=&(\cD\cup\cI_{\text{o}}\cup\cI_{\text{n}})(Q)
  \end{align}
\end{subequations}
where $\kD(Q)$ is the set of externally dead edges, we have
\begin{subequations}
  \begin{align}
    (\cD\cup\kD)(Q)=&(\cD\cup\kD)(\varphi(Q)),\\
    (\cI_{\text o}\cup\cE_{\text o})(Q)=&(\cI_{\text o}\cup\cE_{\text
      o})(\varphi(Q)),\\
     (\cI_{\text n}\cup\cE_{\text n})(Q)=&(\cI_{\text n}\cup\cE_{\text
       n})(\varphi(Q)).
  \end{align}
\end{subequations}
And more precisely,
\begin{align}
  (\cD\cup\cI_{\text{n}})(\varphi(Q))=&\big
  [(\cD\cup\cI_{\text{n}})(Q)\setminus
  E'\big]\cup\big[(\kD\cup\cE_{\text{n}})(Q)\cap E'\big].\label{eq:DInphiQ}\\
  \intertext{Also remember that if $G$ is a signed ribbon graph, and
    $E'\subseteq E(G)$, for all $e\in E'$, the sign of $e$ in $G^{E'}$
    is opposite to the sign of $e$ in $G$; see section \cref{sec:partial-duality}. Thus}
  E_{\pm}\big[(\cD\cup\cI_{\text{n}})(\varphi(Q))\big]=&E_{\pm}\big[(\cD\cup\cI_{\text{n}})(Q)\setminus E'\big]\cup E_{\mp}\big[(\kD\cup\cE_{\text{n}})(Q)\cap E'\big].\label{eq:DInphiQsign}\\
  \intertext{Similarly,}
  \cE_{\text{o}}(\varphi(Q))=&\big[\cE_{\text{o}}(Q)\setminus E'\big]\cup \big[\cI_{\text{o}}(Q)\cap E'\big]\label{eq:EophiQ}\\
  E_{\pm}\big[\cE_{\text{o}}(\varphi(Q))\big]=&E_{\pm}\big[\cE_{\text{o}}(Q)\setminus
  E'\big]\cup E_{\mp}\big[\cI_{\text{o}}(Q)\cap E'\big],\label{eq:EophiQsign}\\
  \cI_{\text{o}}(\varphi(Q))=&\big[\cI_{\text{o}}(Q)\setminus E'\big]\cup\big[\cE_{\text{o}}(Q)\cap E'\big]\label{eq:IophiQ}\\
  E_{\pm}\big[\cI_{\text{o}}(\varphi(Q))\big]=&E_{\pm}\big[\cI_{\text{o}}(Q)\setminus E'\big]\cup E_{\mp}\big[\cE_{\text{o}}(Q)\cap E'\big].\label{eq:IophiQsign}
\end{align}
Before concluding our proof, we need to relate the number of faces of $F_{\cD\cup\cI_{\text{n}}(\varphi(Q))}\in\cS_{G^{E'}}$ to the number of faces of $F_{\cD\cup\cI_{\text{n}}(Q)}\in\cS_{G}$.
\begin{align}
  f(F_{\cD\cup\cI_{\text{n}}(\varphi(Q))})=&v(F_{\cD\cup\cI_{\text{n}}(\varphi(Q))}^\star)=v\big((G^{E'})^{\cD\cup\cI_{\text{n}}(\varphi(Q))}\big)\\
  \intertext{But $\cD\cup\cI_{\text{n}}(\varphi(Q))=\big(E'\Delta E(Q)\big)\setminus \big((\cI_{\text{o}}(Q)\setminus E')\cup(\cE_{\text{o}}(Q)\cap E')\big)$, using $E(Q)=(\cD\cup\cI_{\text{o}}\cup\cI_{\text{n}})(Q)$, so}
  f(F_{\cD\cup\cI_{\text{n}}(\varphi(Q))})=&v\big((G^{E'})^{\lbt E'\Delta E(Q)\rbt\setminus \lbt(\cI_{\text{o}}(Q)\setminus E')\cup(\cE_{\text{o}}(Q)\cap E')\rbt}\big)\nonumber\\
  =&v\big((G^{E(Q)})^{\lbt(\cI_{\text{o}}(Q)\setminus E')\cup(\cE_{\text{o}}(Q)\cap E')\rbt}\big)=1+|\cI_{\text{o}}(Q)\setminus E'|+|\cE_{\text{o}}(Q)\cap E'|,\\
  \intertext{thanks to the fact that the edges in $\cI_{\text{o}}(Q)\cup\cE_{\text{o}}(Q)$ do not cross each other in $G^{E(Q)}$. With the same kind of reasoning, we get}
  f(F_{\cD\cup\cI_{\text{n}}(Q)})=&v\big((G^{E(Q)})^{\cI_{\text{o}}(Q)}\big)=1+|\cI_{\text{o}}(Q)|=1+|\cI_{\text{o}}(Q)\setminus E'|+|\cI_{\text{o}}(Q)\cap E'|\\
  \intertext{and obtain}
  f(F_{\cD\cup\cI_{\text{n}}(\varphi(Q))})=&f(F_{\cD\cup\cI_{\text{n}}(Q)})-|\cI_{\text{o}}(Q)\cap E'|+|\cE_{\text{o}}(Q)\cap E'|.\label{eq:fQQprime}
\end{align}
We are now ready to perform the last computation of this proof. We define $Q'\defi\varphi(Q)$ and $N_{G^{E'}}(Q';1,\balpha,c)\fide\alpha^{D_{\alpha}}c^{d_{c}}(1+\alpha c)^{D_{1}}(\alpha+c)^{D_{2}}$ with
\begin{subequations}
  \begin{align}
    D_{\alpha}=&E_{-}(G^{E'})\cup
    E_{+}(\cD\cup\cI_{\text{n}}(Q'))\setminus\big(E_{-}(\cD\cup\cI_{\text{n}}(Q'))\cup
    E_{-}(\cI_{\text{o}}\cup\cE_{\text{o}}(Q'))\big),\\
    d_{c}=&f(F_{\cD\cup\cI_{\text{n}}(Q')})-|\cI_{\text{o}}(Q')|,\\
    D_{1}=&E_{+}(\cE_{\text{o}}(Q'))\cup E_{-}(\cI_{\text{o}}(Q')),\\
    D_{2}=&E_{-}(\cE_{\text{o}}(Q'))\cup E_{+}(\cI_{\text{o}}(Q')).
  \end{align}
\end{subequations}  
Now, using equations \eqref{eq:EophiQsign} and \eqref{eq:IophiQsign},
\begin{align}
  D_{\alpha}=&\big(E_{-}(G)\cup E_{+}(E')\big)\setminus E_{-}(E')\cup E_{+}\big[(\cD\cup\cI_{\text{n}})(Q)\setminus E'\big]\cup E_{-}\big[(\kD\cup\cE_{\text{n}})(Q)\cap E'\big]\nonumber\\
  &\setminus\big(E_{-}\big[(\cD\cup\cI_{\text{n}})(Q)\setminus
  E'\big]\cup E_{+}\big[(\kD\cup\cE_{\text{n}})(Q)\cap
  E'\big]\nonumber\\
  &\cup E_{-}\big[(\cI_{\text{o}}\cup\cE_{\text{o}})(Q)\setminus E'\big]\cup
  E_{+}\big[(\cI_{\text{o}}\cup\cE_{\text{o}})(Q)\cap E'\big]\big).\\
  \intertext{As $E'=E'\cap E(G)=E'\cap\big[(\cD\cup\cI_{\text
      n}\cup\cI_{\text o}\cup\cE_{\text o}\cup\kD\cup\cE_{\text
      n})(Q)\big]$, we have $(\kD\cup\cE_{\text n})(Q)\cap
    E'=\big[E'\setminus(\cD\cup\cI_{\text
      n})(Q)\big]\setminus\big[(\cI_{\text o}\cup\cE_{\text o})(Q)\cap E'\big]$,
  and}
  D_{\alpha}=&E_{-}(G)\cup E_{+}(E')\cup
  E_{+}\big[(\cD\cup\cI_{\text{n}})(Q)\setminus E'\big]\cup E_{-}\big[E'\setminus (\cD\cup\cI_{\text{n}})(Q)\big]\nonumber\\
  &\setminus\big(E_{+}\big[E'\setminus
  (\cD\cup\cI_{\text{n}})(Q)\big]\cup E_{-}(E')\cup E_{-}\big[(\cI_{\text{o}}\cup\cE_{\text{o}})(Q)\cap E'\big]\nonumber\\
  &\cup E_{-}\big[(\cD\cup\cI_{\text{n}})(Q)\setminus E'\big]\cup E_{-}\big[(\cI_{\text{o}}\cup\cE_{\text{o}})(Q)\setminus E'\big]\big)\\
  =&E_{-}(G)\cup
  E_{+}(\cD\cup\cI_{\text{n}}(Q))\setminus\big(E_{-}(\cD\cup\cI_{\text{n}}(Q))\cup
  E_{-}(\cI_{\text{o}}\cup\cE_{\text{o}}(Q))\big).\\
  \nonumber\\
  \intertext{Using \cref{eq:EophiQsign,eq:IophiQ,eq:IophiQsign,eq:fQQprime},}
  d_{c}=&f(F_{\cD\cup\cI_{\text{n}}(Q)})-|\cI_{\text{o}}(Q)\cap E'|+|\cE_{\text{o}}(Q)\cap E'|-\labs\cI_{\text{o}}(Q)\setminus E'\rabs-\labs\cE_{\text{o}}(Q)\cap E'\rabs\\
  =&f(F_{\cD\cup\cI_{\text{n}}(Q)})-|\cI_{\text{o}}(Q)|,\\
  \nonumber\\
  D_{1}=&E_{+}\big[\cE_{\text{o}}(Q)\setminus E'\big]\cup
  E_{-}\big[\cI_{\text{o}}(Q)\cap E'\big]\cup
  E_{-}\big[\cI_{\text{o}}(Q)\setminus E'\big]\cup E_{+}\big[\cE_{\text{o}}(Q)\cap E'\big]\\
  =&E_{+}(\cE_{\text{o}}(Q))\cup E_{-}(\cI_{\text{o}}(Q)),\\
\nonumber\\
D_{2}=&E_{-}\big[\cE_{\text{o}}(Q)\setminus E'\big]\cup
E_{+}\big[\cI_{\text{o}}(Q)\cap E'\big]\cup
E_{+}\big[\cI_{\text{o}}(Q)\setminus E'\big]\cup E_{-}\big[\cE_{\text{o}}(Q)\cap E'\big]\\
=&E_{-}(\cE_{\text{o}}(Q))\cup E_{+}(\cI_{\text{o}}(Q)).
\end{align}
This proves that $N_{G}(Q;1,\balpha,c)=N_{G^{E'}}(\varphi(Q);1,\balpha,c)$, meaning that the bijection (\cref{eq:Bijection}) conserves independently each of the terms (i.e.\@ the $N(Q)'s$) of the quasi-tree expansion. This implies, of course, $Z(G;1,\balpha,c)=Z(G^{E'};1,\balpha,c)$.
\end{proof}

The preceding lemma shows that, given a ribbon graph $G$, a subset of
edges $E'$ and a quasi-tree $Q\in\cQ_{G}$, there exists a quasi-tree
$Q'\in\cQ_{G^{E'}}$ such that
$N_{G}(Q;1,\balpha,c)=N_{G^{E'}}(Q';1,\balpha,c)$. The subribbon graph
$Q'$ is such that $E(Q')=\varphi(Q)$. But we can also invert the
logic: given a ribbon graph $G$, a quasi-tree $Q\in\cQ_{G}$ and a
subset of edges $A\subseteq E(G)$, there exists a subset $E'$ such
that the spanning subribbon graph $Q'$ with the property that
$E(Q')=A$ is a quasi-tree in $G^{E'}$ and
$N_{G}(Q;1,\balpha,c)=N_{G^{E'}}(Q';1,\balpha,c)$. Whatever subset $A$
we choose, the bijection ensures that $Q'$ is a quasi-tree in
$G^{E'}$. This means that we can fix $A$ and deduce the set $E'$. A
very simple case is $A=\emptyset$: given $Q\in\cQ_{G}$, in which
partial dual of $G$ is the empty set a quasi-tree? The answer is given by the bijection $\varphi$:
\begin{align}
  E'\Delta E(Q)=\emptyset\iff E'=E(Q).
\end{align}
And we get: for any quasi-tree $Q\in\cQ_{G}$, $N_{G}(Q;1,\balpha,c)=N_{G^{E(Q)}}(F_{\emptyset};1,\balpha,c)$. In that case, $Q'$ having no edge, the live (or dead) edges are necessarily external. Let us define $\mathbf{\cL_{\text{\textbf{o}}}(Q)}\defi\lb\text{orientable live edges of $G^{E(Q)}$ with respect to $F_{\emptyset}$}\rb$. 
\begin{lemma}
  \label{lem:q1Exp}
  For any ribbon graph $G$, the quasi-tree expansion for $Z_{s}$ at $q=1$ can be rewritten as follows:
  \begin{align}
    Z_{s}(G;1,\balpha,c)=&c\sum_{Q\in\cQ_{G}}\alpha^{e_{-}(G^{E(Q)})}(1+\alpha c)^{e_{+}(\cL_{\text{o}}(Q))}(1+c/\alpha)^{e_{-}(\cL_{\text{o}}(Q))}.\label{eq:QTExp2Z}
  \end{align}
\end{lemma}

\section{The Kauffman bracket of a virtual link diagram}
\label{sec:kauffm-brack-virt}

In \citep{Chmutov2007aa}, S.~Chmutov unified several Thistlethwaite
like theorems
\citep{Thistlethwaite1987aa,Kauffman1989aa,Dasbach2008aa,Chmutov2007ab,Chmutov2008aa}
(that is theorems relating link and (ribbon) graph polynomials). He
proved that the Kauffman bracket of a virtual link diagram $L$ equals
(an evaluation of) the signed \BRp{} of a certain ribbon graph
$G_{L}^{\ks}$; see \cref{eq:AlternatBRCorresp}. The latter is
constructed from a state $\ks$ of $L$; see below and/or \citep{Chmutov2007aa}. The equality is true for \emph{any} state $\ks$.
\begin{align}
  [L](A,B,d)=A^{n(G_{L})}B^{r(G_{L})}d^{k(G_{L})-1}R_{s}(G^{\ks}_{L};\frac{Ad}{B}+1,\frac{Bd}{A},\frac 1d).\label{eq:AlternatBRCorresp}
\end{align}
The new \emph{partial} duality of S.~Chmutov ensures the independence
of the right hand side of \eqref{eq:AlternatBRCorresp} with respect to the state $\ks$.\\

In the previous sections, we obtained a quasi-tree expansion for the \BRp. Thanks to equation \eqref{eq:AlternatBRCorresp}, we can obviously get such an expansion for the Kauffman bracket. Nevertheless, this expansion would be expressed in terms of parameters (number of vertices, edges etc) of the (subribbon graphs of the) ribbon graph $G_{L}^{\ks}$ associated with the state $\ks$ of $L$. Here we would like to get a new expansion for the Kauffman bracket, directly expressed in terms of the parameters of the states of $L$.\\

Combining equations (\cref{eq:AlternatBRCorresp}) and (\cref{eq:CorrespZRs}), we get
\begin{align}
  [L](A,B,d)=&A^{e(G_{L}^{\ks})}d^{-1}Z(G_{L}^{\ks};1,B/A,d).\label{eq:CorrespZKauff}
\end{align}
Now, using the expansion (\cref{eq:QTExp2Z}),
\begin{equation}
\begin{split} [L](A,B,d)=A^{e(G_{L}^{\ks})}\sum_{Q\in\cQ_{G_{L}^{\ks}}}\big(B/A\big)&^{e_{-}((G_{L}^{\ks})^{E(Q)})}\\
  &\big(1+Bd/A\big)^{e_{+}(\cL_{\text{o}}(Q))}\big(1+Ad/B\big)^{e_{-}(\cL_{\text{o}}(Q))}.\label{eq:ConnExpStep1}
\end{split}
\end{equation}
Let us now translate this expression into pure ``knot theoretical''
terms. For this, we need to recall how the ribbon graph $G_{L}^{\ks}$
is built, out of the state $\ks$ of the virtual link diagram $L$. The
state $\ks$ consists in a set of (possibly nested) circles, called
state circles, which writhe at the virtual crossings; see \cref{VirtWHState} for an example. For each state of $L$, each
classical crossing is resolved i.e. at each classical crossing, one
performs either an $A$- or a $B$-splitting; see \cref{fig:ABsplittings}. Now each
resolved crossing consists of two parallel strands. In the vicinity
of each former classical crossing, place one arrow on each of these
strands, pointing in opposite directions, figure
\cref{PlaceArrows}. Label these two arrows with a common name and a
sign: $+$ if the former crossing has been resolved by an $A$-splitting
and $-$ otherwise. Then pull the state circles apart, untwisting them if
needed. The result is the combinatorial representation of the ribbon
graph $G_{L}^{\ks}$, \cref{CombReprGLs}.
\begin{figure}[!htp]
  \centering
  \subfloat[An $A$-splitting]{{\label{Asplitt}}%
    \begin{tikzpicture}[decoration=snake]
      \foreach \u in {3} 
      {
        \node (un) at (0,0) {\includegraphics[scale=.4]{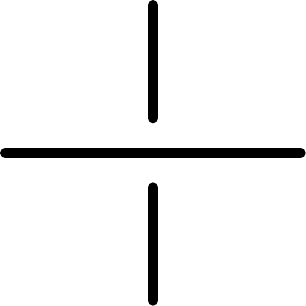}};
        \node (deux) at (\u,0) {\includegraphics[scale=.4]{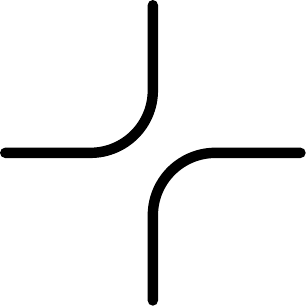}};       
        \draw [thick,->,decorate,segment length=12] (un.east)--(deux.west) 
        ;
      }
    \end{tikzpicture}
  }\hspace{2cm}
   \subfloat[A $B$-splitting]{{\label{Asplitt}}%
    \begin{tikzpicture}[decoration=snake]
      \foreach \u in {3} 
      {
        \node (un) at (0,0) {\includegraphics[scale=.4]{knot-15}};
        \node (deux) at (\u,0) {\includegraphics[scale=.4]{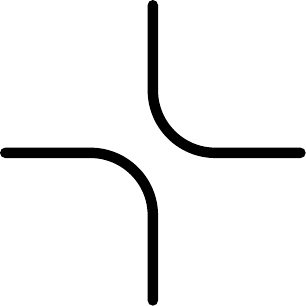}};       
        \draw [thick,->,decorate,segment length=12] (un.east)--(deux.west) 
        ;
      }
    \end{tikzpicture}
  }
  \caption{$A$- and $B$-splittings}
  \label{fig:ABsplittings}
\end{figure}
\begin{figure}[!htp]
  \centering
  \subfloat[A virtual version of the Whitehead link]{{\label{VirtWH}}\includegraphics[scale=.4]{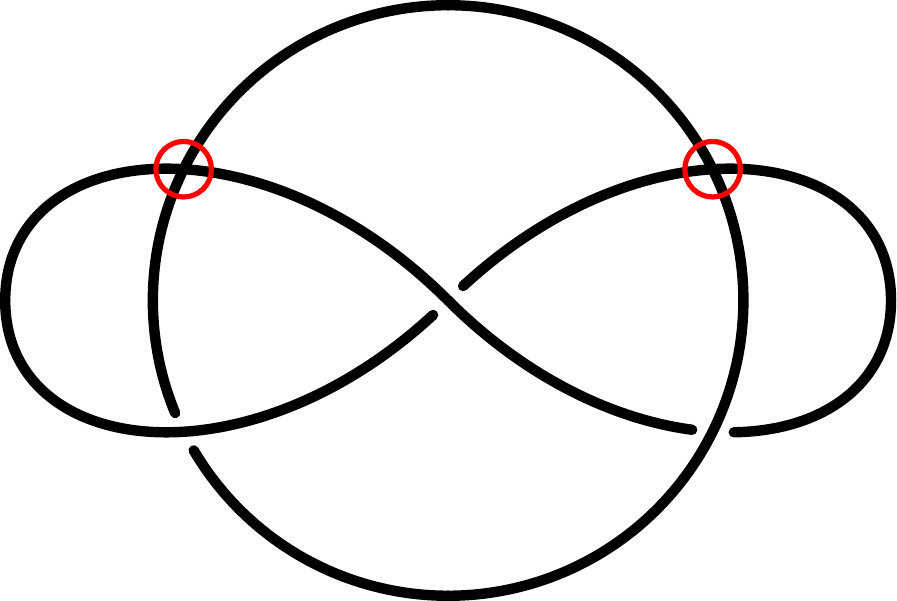}}\hspace{2cm}
  \subfloat[A state $\ks$ of the link of \cref{VirtWH}]{{\label{VirtWHState}}\includegraphics[scale=.4]{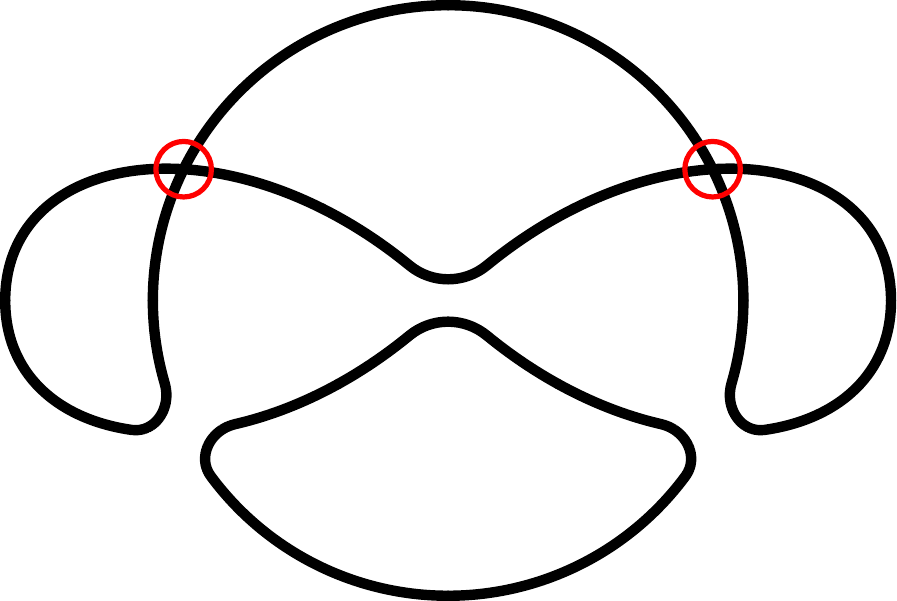}}\\
  \subfloat[Placing the edge-arrows]{{\label{PlaceArrows}}\includegraphics[scale=.4]{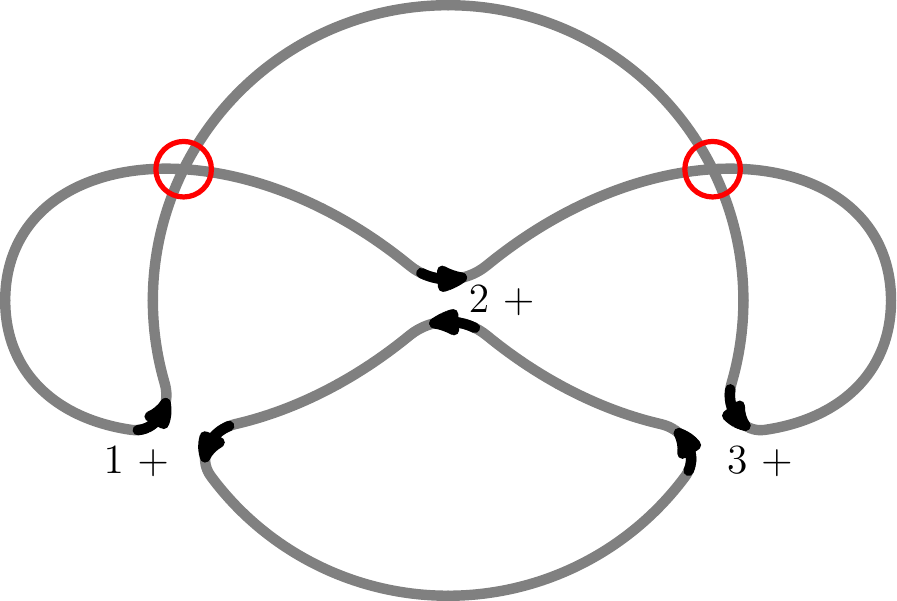}}\\
  \subfloat[Combinatorial representation of $G_{L}^{\ks}$]{{\label{CombReprGLs}}\includegraphics[scale=.5]{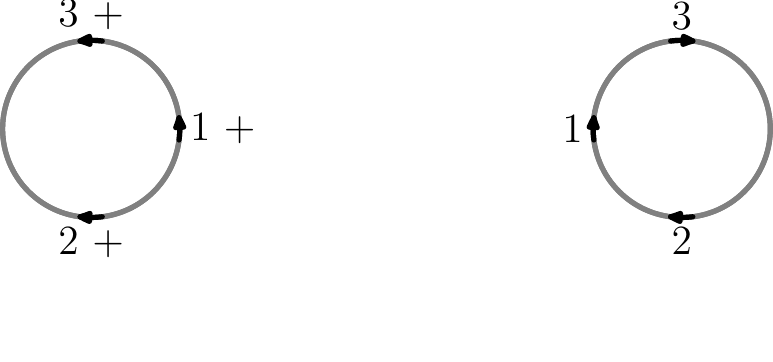}}\hspace{2cm}
  \subfloat[The ribbon graph $G_{L}^{\ks}$]{{\label{GLsEx}}\includegraphics[scale=.5]{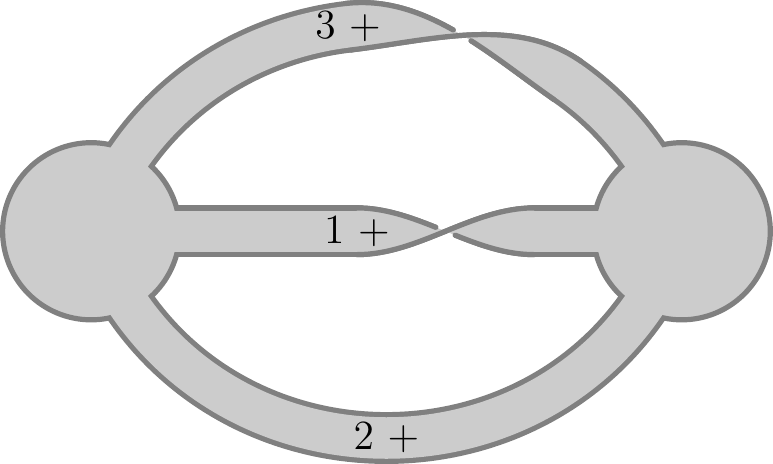}}
  \caption{Construction of a $G_{L}^{\ks}$}
  \label{fig:GLsConstruct}
\end{figure}\\

The \Kb{} and the \BRp{} being related by equation (\cref{eq:AlternatBRCorresp}), there is a one-to-one correspondence between the states of a virtual link diagram $L$ and the spanning subribbon graphs of $G_{L}^\ks$. First of all, given the construction of $G_{L}^\ks$, there is a bijection $\kappa$ between the crossings of $L$ and the edges of $G_{L}^\ks$. Then, writing $C_{\ks'\neq\ks}$ for the set of crossings which are resolved differently in $\ks$ and $\ks'$, the bijection between the states of $L$ and the subribbon graphs of $G_{L}^{\ks}$ is
\begin{align}
  \sigma: \ks'&\mapsto \sigma(\ks')=F\in\cS_{G_{L}^\ks} \text{ s.t.\@ }E(F)=\kappa(C_{\ks'\neq\ks}).
\end{align}
Another crucial point, noticed by S.~Chmutov \citep{Chmutov2007aa}, is the fact that, given two states $\ks$ and $\ks'$, the ribbon graphs $G_{L}^{\ks}$ and $G_{L}^{\ks'}$ are dual to each other with respect to $\kappa(C_{\ks\neq \ks'})$:
\begin{align}
  G_{L}^{\ks'}=&(G_{L}^{\ks})^{\kappa(C_{\ks\neq \ks'})}.\label{eq:DualStates}
\end{align}
This allows us to understand to which state a quasi-tree
corresponds. Let us consider a state $\ks'$ with only one state
circle, hereafter called a \textbf{connected state}. The ribbon graph
$G_{L}^{\ks'}$ has only one vertex. But, by equation
(\cref{eq:DualStates}), the partial dual of $G_{L}^{\ks}$ with respect
to $\kappa(C_{\ks\neq \ks'})$ has only one vertex, meaning that the
subribbon graph of $G_{L}^{\ks}$, the edge-set of which is
$\kappa(C_{\ks\neq \ks'})$, is a quasi-tree. In contrast, a quasi-tree $Q$ defines a unique connected state $\ks'$ by the equation $E(Q)=\kappa(C_{\ks\neq \ks'})$. Then the set of quasi-trees of $G_{L}^{\ks}$ corresponds to the set of connected states of $L$.\\

To complete our translation of the expansion (\cref{eq:ConnExpStep1}),
we now explain to which crossings the orientable live edges
correspond. Given a quasi-tree $Q$ of $G_{L}^{\ks}$,
$\cL_{\text{o}}(Q)$ is the set of orientable live edges of
$(G_{L}^\ks)^{E(Q)}$ with respect to $F_{\emptyset}$. But there exists
a unique connected state $\ks'=\sigma^{-1}(Q)$ such that
$(G_{L}^\ks)^{E(Q)}=G_{L}^{\ks'}$. The state circle of $\ks'$ is the
boundary of the vertex of $G_{L}^{\ks'}$. Then, to determine whether a
crossing is live with respect to $\ks'$, we mark the resolved
crossings in $\ks'$, as in the example of \cref{PlaceArrows}. What we
get is a (possibly twisting) circle with $2n$ marks ($n=$ number of
crossings of $L$), labelled with $n$ different names. To decide
whether a crossing $c$ is live or not, turn around the state circle of
$\ks'$, starting at one of the two marks corresponding to $c$. Before
reaching the second mark of $c$, we meet other labels. A label met
twice, called paired, corresponds to an edge in $G_{L}^{\ks}$ which
does not cross $\kappa(c)$ in $G^{\ks'}_{L}$. In contrast, a label met only once, called single, corresponds to an edge crossing $\kappa(c)$. Then $c$ is live if, from one mark of $c$ to the other, we meet no single lower-ordered label. Otherwise, it is dead.

Finally, a crossing $c$ is orientable with respect to a connected state $\ks'$ if, from one mark of $c$ to the other, we pass through virtual crossings an even number of times. For example, with respect to the connected state of \cref{fig:ConnectedStateEx}, crossings $1$ and $3$ are orientable, whereas crossing $2$ is non-orientable.
\begin{figure}[!htp]
  \centering
  \includegraphics[scale=.4]{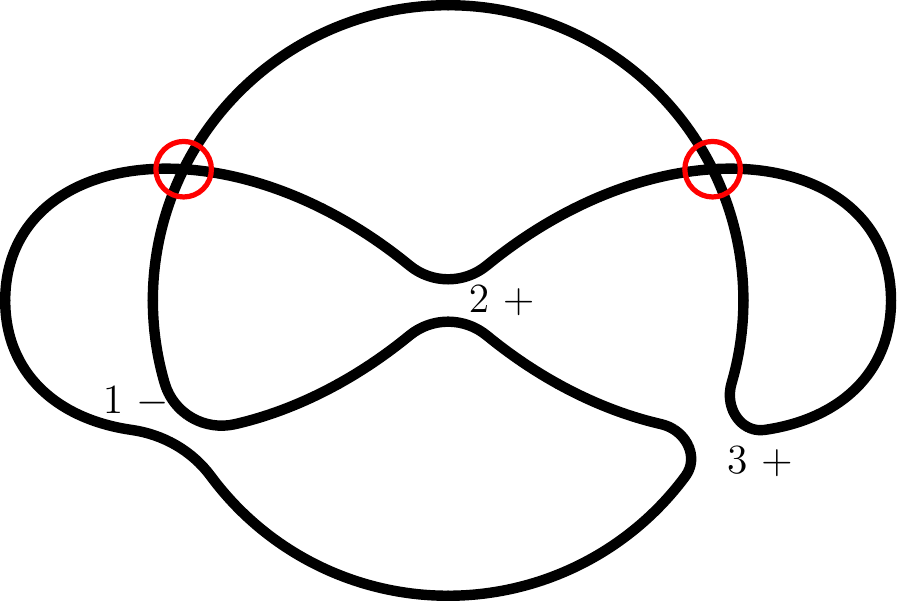}
  \caption{A connected state}
  \label{fig:ConnectedStateEx}
\end{figure}
\newpage
To sum up, we have:
\begin{itemize}
\item each quasi-tree $Q\in\cS_{G_{L}^{\ks}}$ corresponds to a connected state $\ks'=\sigma^{-1}(Q)$,
\item $e(G_{L}^{\ks})=n(L)$ is the number of crossings of $L$,
\item $e_{+}((G_{L}^{\ks})^{E(Q)})=a_{L}(\ks')$ is the number of $A$-splittings of $\ks'$,
\item $e_{-}((G_{L}^{\ks})^{E(Q)})=b_{L}(\ks')$ is the number of $B$-splittings of $\ks'$,
\item $e_{+}(\cL_{\text{o}}(Q))\fide|\cL_{\text{o}}^{a}(\ks')|$ is the
  number of live orientable crossings resolved by $A$-splittings in $\ks'$,
\item $e_{-}(\cL_{\text{o}}(Q)) \fide|\cL_{\text{o}}^{b}(\ks')|$ is the
  number of live orientable crossings resolved by $B$-splittings in $\ks'$.
\end{itemize}
So we get:
\begin{lemma}[Connected state expansion]
  \label{lem:ConnExpKb}
  Let $L$ be a virtual link diagram. For any order for the crossings of $L$, the \Kb{} can be rewritten as
  \begin{equation}
     [L](A,B,d)=\sum_{\substack{\text{connected}\\\text{states $\ks'$ of $L$}}}A^{a_{L}(\ks')}B^{b_{L}(\ks')}
      \big(1+Bd/A\big)^{|\cL_{\text{o}}^{a}(\ks')|}\big(1+Ad/B\big)^{|\cL_{\text{o}}^{b}(\ks')|}.\label{eq:ConnExp}
 \end{equation}
\end{lemma}

\vspace{.7cm}

\paragraph{Acknowledgements}
\label{sec:acknowledgements}

I thank A.~Champanerkar, I.~Kofman and N.~Stoltzfus for having
explained to me some details of their inspiring work
\citep{Champanerkar2007aa}. I also acknowledge the anonymous referee's
work that led to this improved version.

\newpage
\appendix

\section{Examples}
\label{sec:example}

\subsection{The \BRp}
\label{sec:brp-1}

We give here an example of our quasi-tree expansion of the signed \BRp{} of not necessarily orientable ribbon graphs. We choose the non-orientable signed ribbon graph $G$ of \cref{fig:NORibbonEx}.
\begin{figure}[!htp]
  \centering
  \includegraphics[scale=.3]{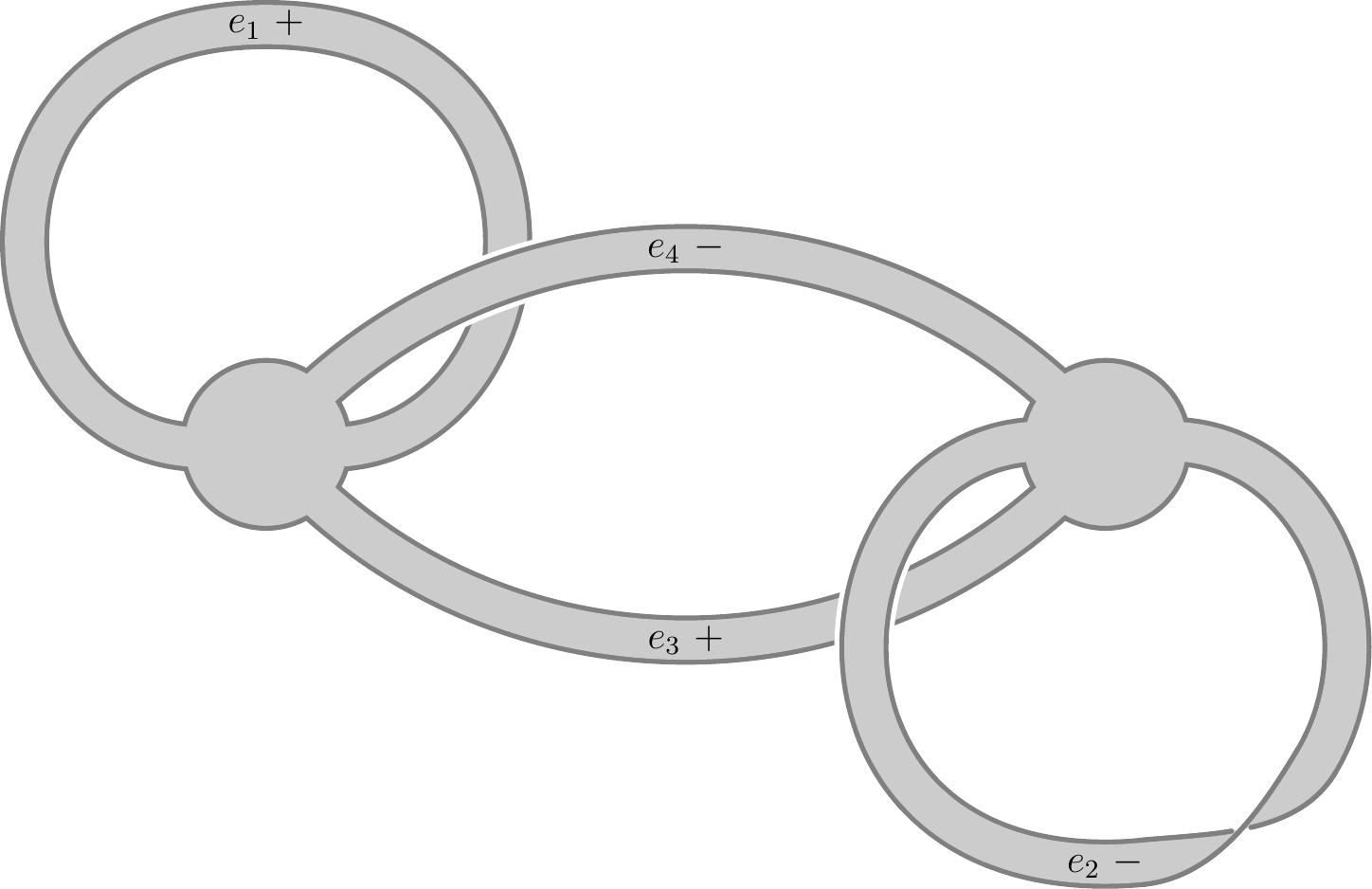}
  \caption{A non-orientable signed ribbon graph}
  \label{fig:NORibbonEx}
\end{figure}
According to equation (\cref{eq:SignedBRpDef}), the signed \BRp{} of $G$ is
\begin{multline}
  R_{s}(G;x+1,y,z)=1+3y+y^{2}+xz+yz+2xyz+y^{2}z\\
  +xy^{2}z+xyz^{2}+y^{2}z^{2}+xy^{2}z^{3}+x^{-1}y+x^{-1}y^{2}.\label{eq:SubExpEx}
\end{multline}
We now check that the quasi-tree expansion (\cref{eq:SBRpExp}) gives the same polynomial. To this aim, according to \cref{thm:QTSBRP}, we define
\begin{subequations}
  \begin{align}
    P(G;x,y,z)\defi&\sum_{Q\in\cQ_{G}}N(G,Q)S(G_{Q}),\\
    N(G,Q)\defi&(x^{-1/2}y^{1/2})^{\,e_{-}(G)}x^{e_{-}(\cD(Q)\cup\cI_{\text{n}}(Q))}y^{n(F_{\cD(Q)\cup\cI_{\text{n}}(Q)})-e_{-}(\cD(Q)\cup\cI_{\text{n}}(Q))}\nonumber\\
    &z^{(k-f+n)(F_{\cD(Q)\cup\cI_{\text{n}}(Q)})} (1+x)^{e_{-}(\cE_{\text{o}}(Q))}(1+y)^{e_{+}(\cE_{\text{o}}(Q))},\\
    S(G_{Q})\defi&(x^{1/2}y^{-1/2})^{r(G_{Q})+e_{-}(\cI_{\text{o}}(Q))}Ra(G_{Q};1,x^{-1/2}y^{1/2},x^{1/2}y^{1/2},x^{1/2}y^{1/2}z^{2}).
  \end{align}
\end{subequations}
We want to check that $P(G;x,y,z)=R_{s}(G;x+1,y,z)$. Table \cref{tab:QTreeExpEx} lists the information necessary for computing the polynomial $P$. We get
\begin{multline}
  P(G;x,y,z)=(1+x^{-1})y(1+y)+1+y+(1+x)y(1+y)z\\
  +x(1+y)z+y(1+yz^{2})+xyz^{2}+xy^{2}z^{3}.\label{eq:QTEX}
\end{multline}
We easily see that the right hand sides of equations (\cref{eq:SubExpEx}) and (\cref{eq:QTEX}) are equal.
   \begin{landscape}

      \begin{longtable}{|Sc|Sc||Sc|Sc||Sc|Sc|}
        \caption{Quasi-tree expansion of $R_s(G)$}\\
        \hline
        $\mathbf{E(Q)}$&$\mathbf{G^{E(Q)}}$&$\mathbf{\cI_{\text{o}}(Q)}$,
        $\mathbf{\cI_{\text{n}}(Q)}$, $\mathbf{\cD(Q)}$, $\mathbf{\cE_{\text{o}}(Q)}$&$\mathbf{N(G,Q)}$&$\mathbf{G_{Q}}$&$\mathbf{S(G_{Q})}$\\
        \hline
        \endhead
        \multicolumn{6}{r}{\textsf{Continued on next page}}
        \endfoot
        \hline
        \endlastfoot
        $\set{e_{3}}$&\raisebox{-.5\height}{\includegraphics[scale=.2]{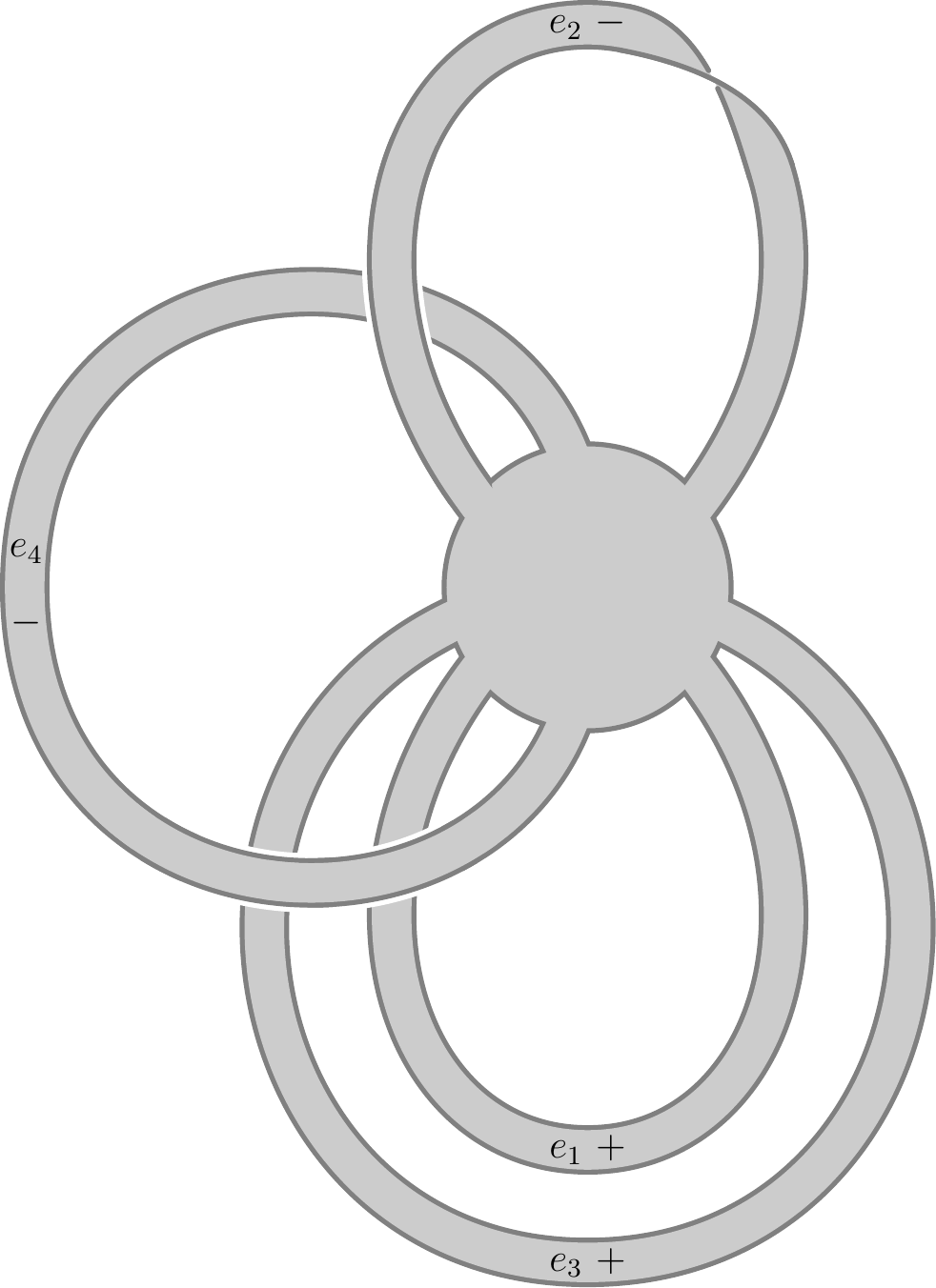}}&$\set{e_{3}}$,
        $\emptyset$, $\emptyset$, $\set{e_{1}}$&$x^{-1/2}y^{1/2}(1+y)$&\raisebox{-.5\height}{\includegraphics[scale=.7]{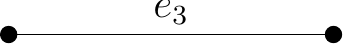}}&$x^{1/2}y^{1/2}+x^{-1/2}y^{1/2}$\\
        $\set{e_{4}}$&\raisebox{-.5\height}{\includegraphics[scale=.2]{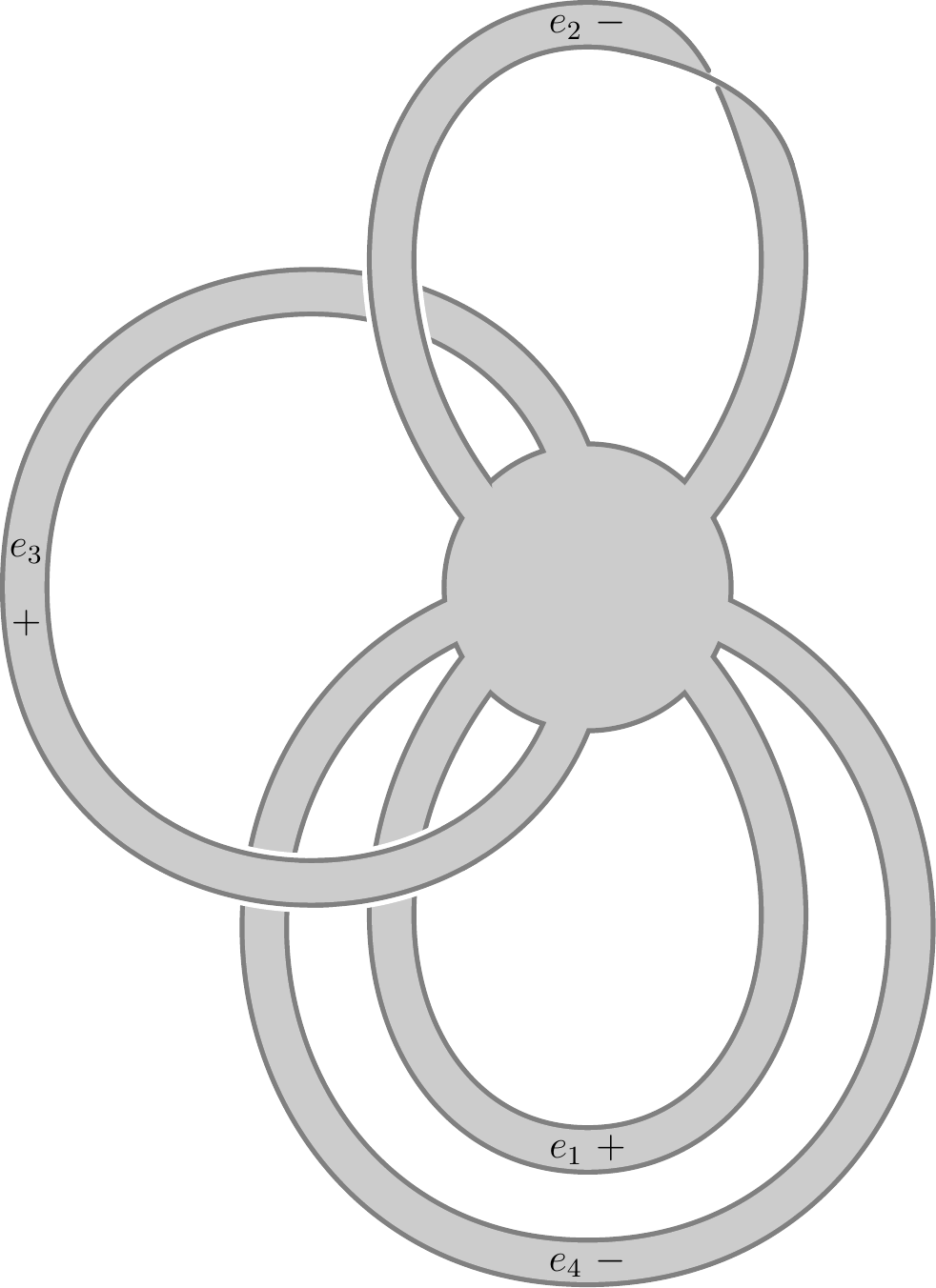}}&$\emptyset$,
        $\emptyset$, $\set{e_{4}}$, $\set{e_{1}}$&$(1+y)$&\raisebox{-.5\height}{\includegraphics[scale=.7]{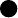}}&$1$\\
        $\set{e_{2},e_{3}}$&\raisebox{-.5\height}{\includegraphics[scale=.2]{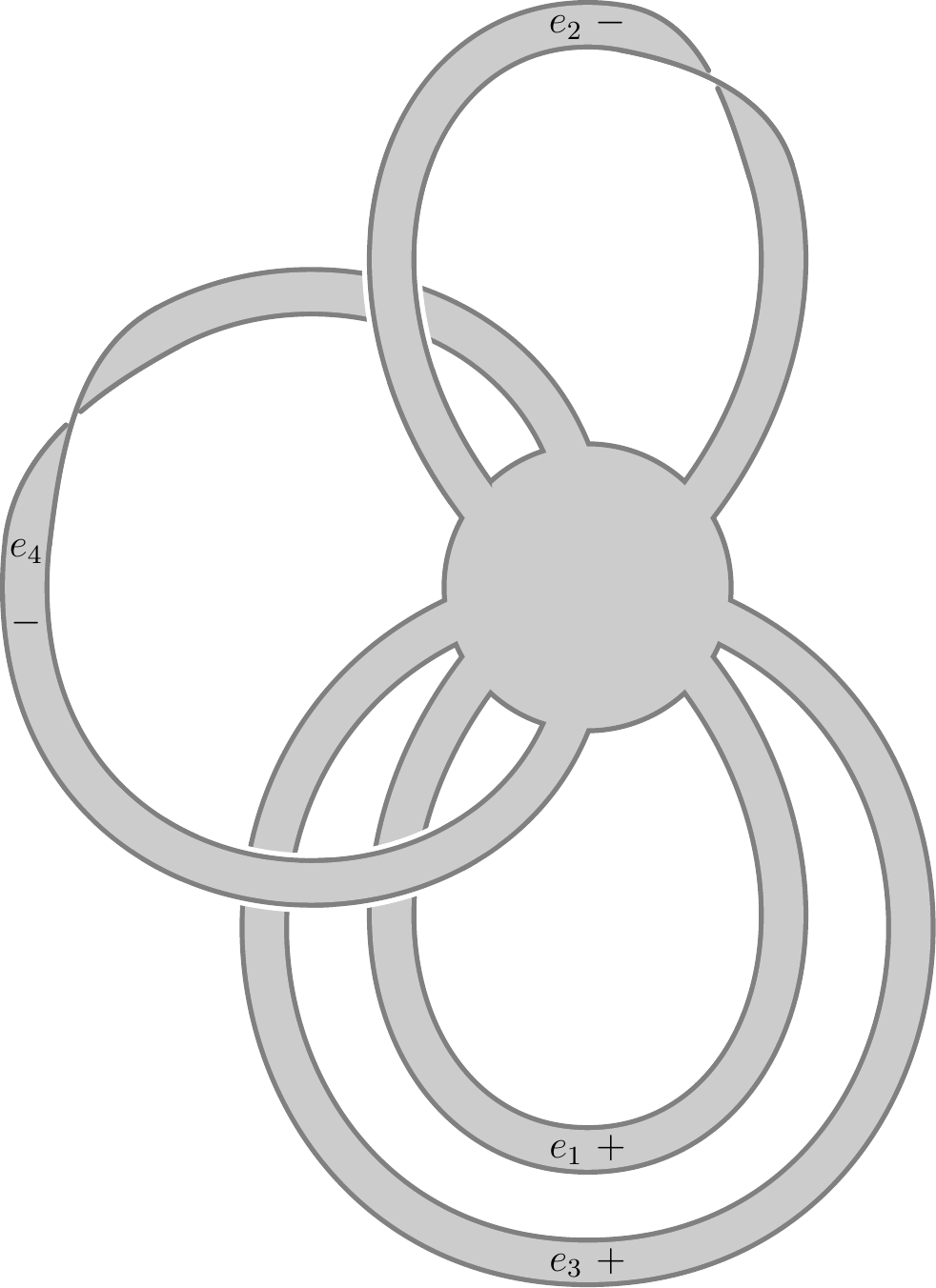}}&$\set{e_{3}}$,
        $\set{e_{2}}$, $\emptyset$, $\set{e_{1}}$&$x^{1/2}y^{1/2}z(1+y)$&\raisebox{-.5\height}{\includegraphics[scale=.7]{QTree-fig-12.pdf}}&$x^{1/2}y^{1/2}+x^{-1/2}y^{1/2}$\\
        $\set{e_{2},e_{4}}$&\raisebox{-.5\height}{\includegraphics[scale=.2]{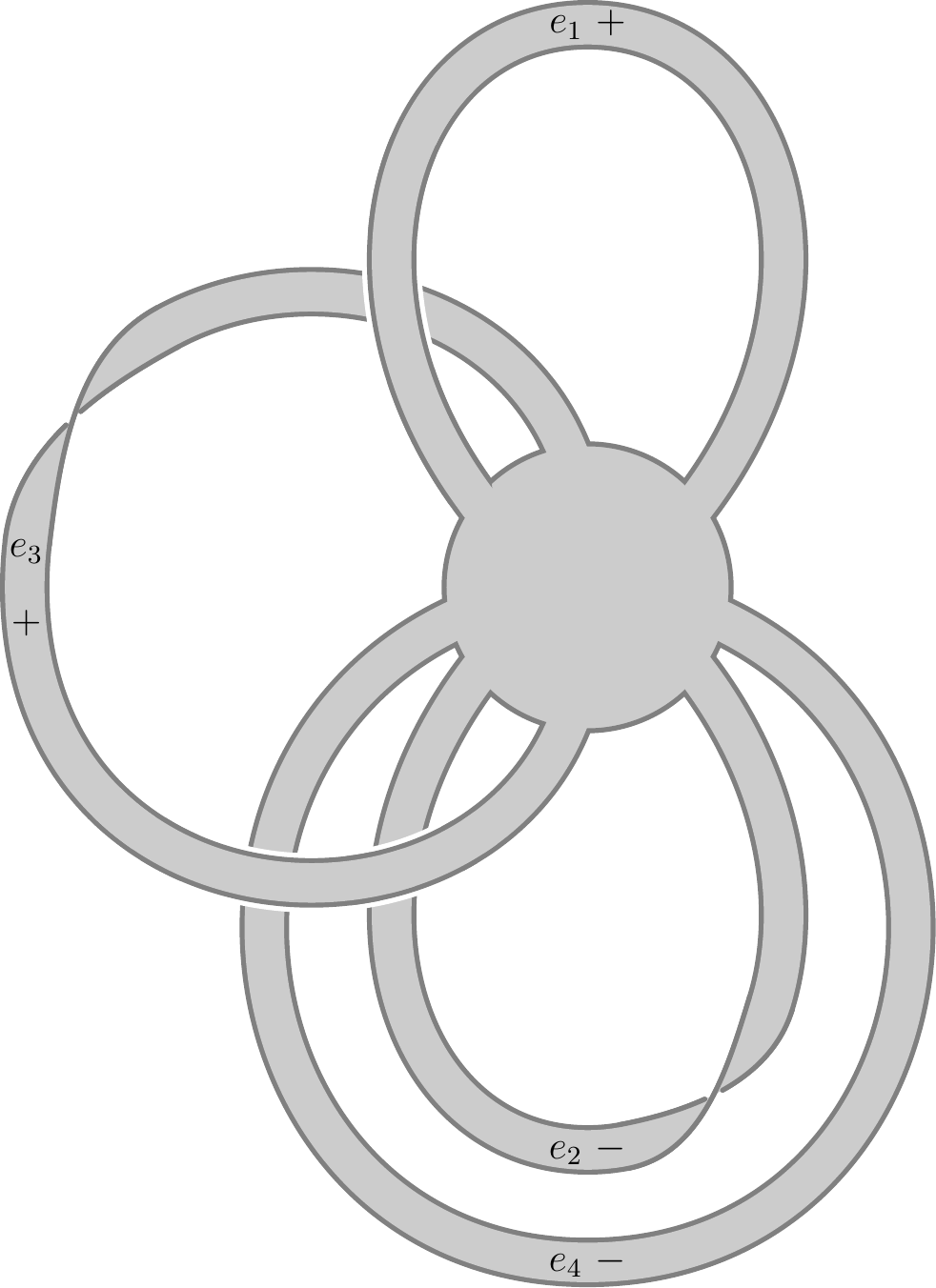}}&$\emptyset$,
        $\set{e_{2}}$, $\set{e_{4}}$, $\set{e_{1}}$&$xz(1+y)$&\raisebox{-.5\height}{\includegraphics[scale=.7]{QTree-fig-13.pdf}}&$1$\\
        $\set{e_{1},e_{3},e_{4}}$&\raisebox{-.5\height}{\includegraphics[scale=.2]{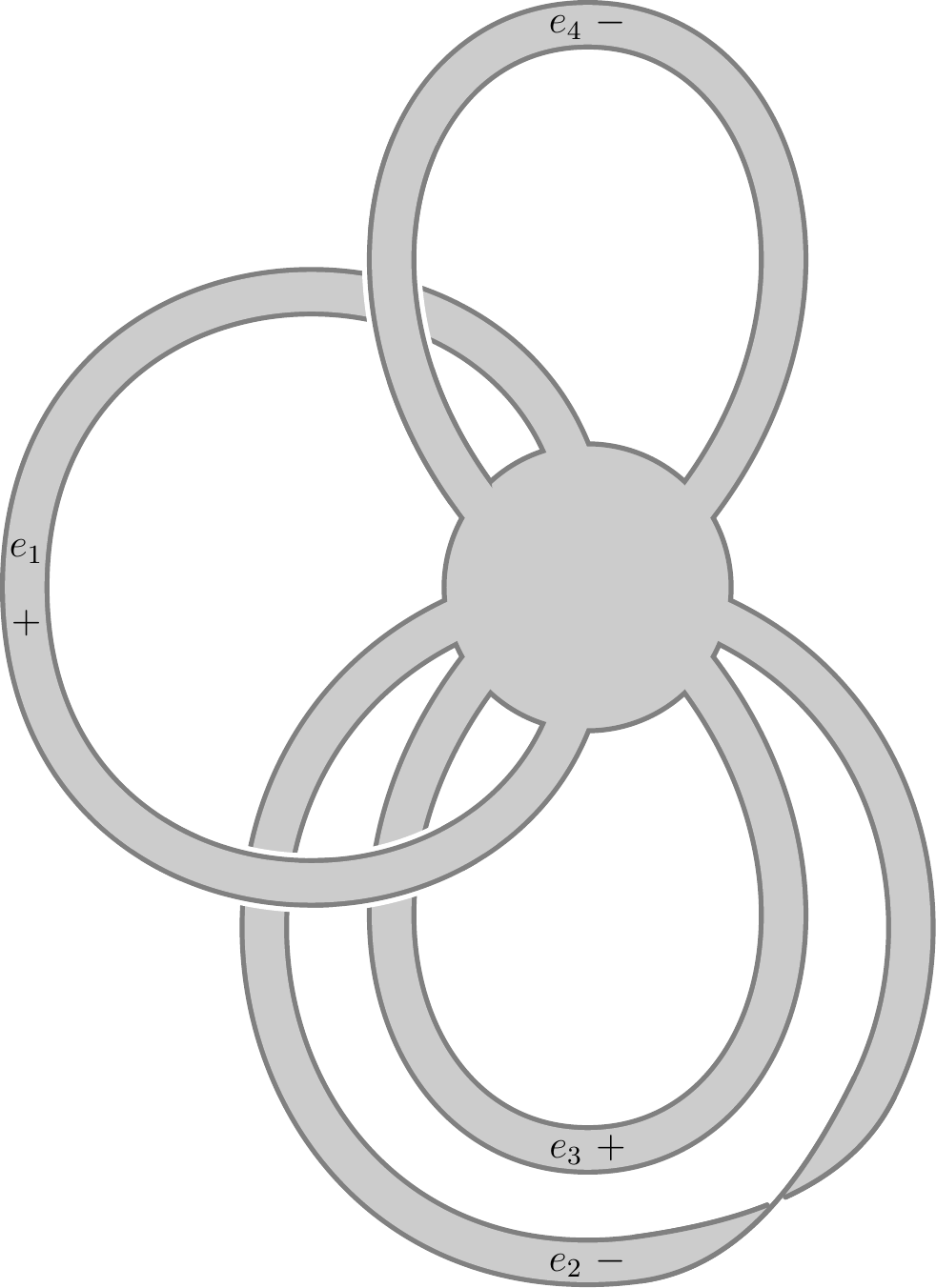}}&$\set{e_{1}}$,
        $\emptyset$, $\set{e_{3},e_{4}}$, $\emptyset$&$y$&\raisebox{-.5\height}{\includegraphics[scale=.7]{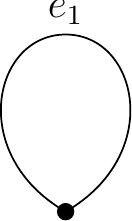}}&$1+yz^{2}$\\
        $\set{e_{2},e_{3},e_{4}}$&\raisebox{-.5\height}{\includegraphics[scale=.2]{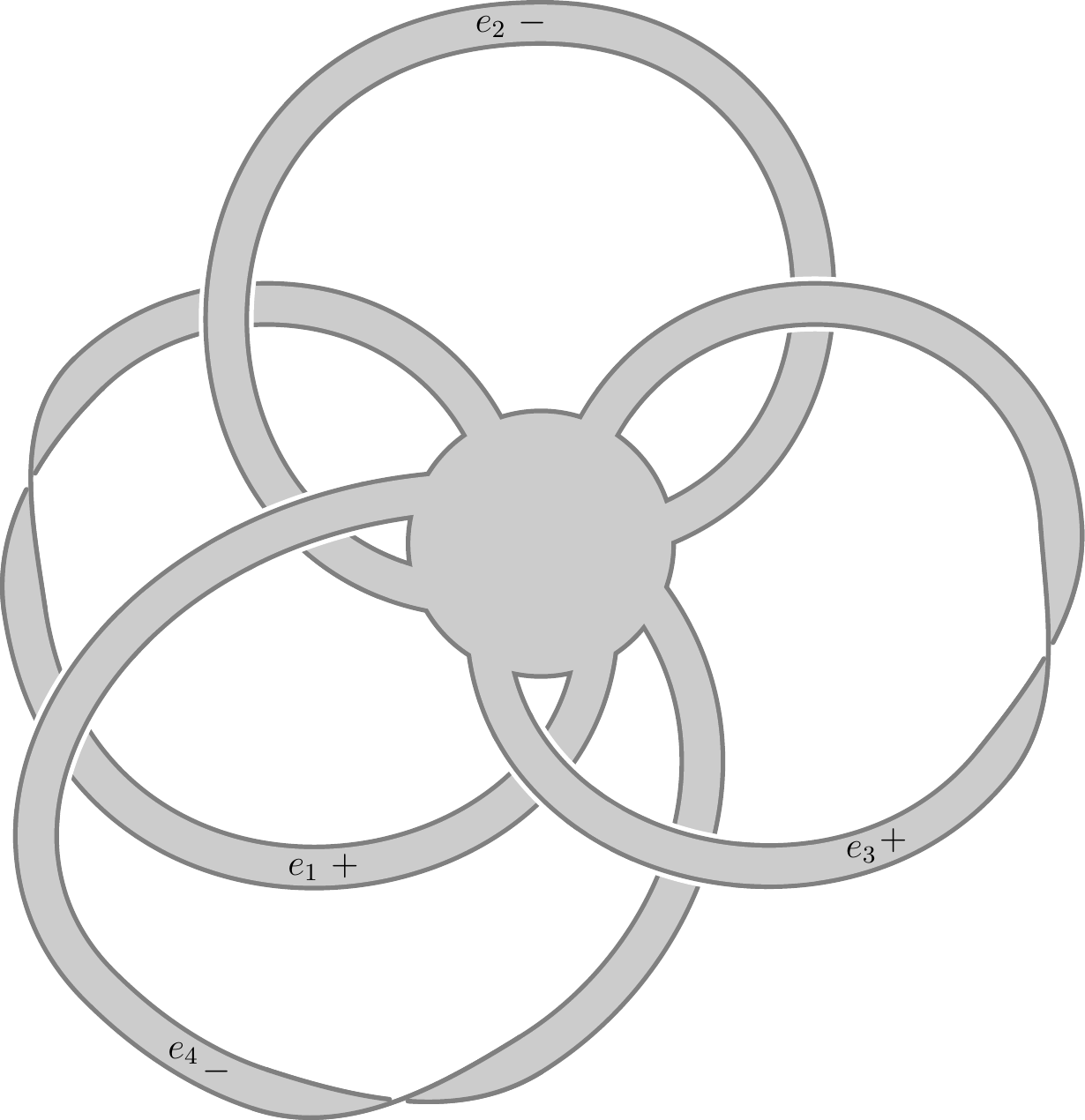}}&$\emptyset$,
        $\emptyset$, $\set{e_{2},e_{3},e_{4}}$, $\emptyset$&$xyz^{2}$&\raisebox{-.5\height}{\includegraphics[scale=.7]{QTree-fig-13.pdf}}&$1$\\
        $\set{e_{1},e_{2},e_{3},e_{4}}$&\raisebox{-.5\height}{\includegraphics[scale=.2]{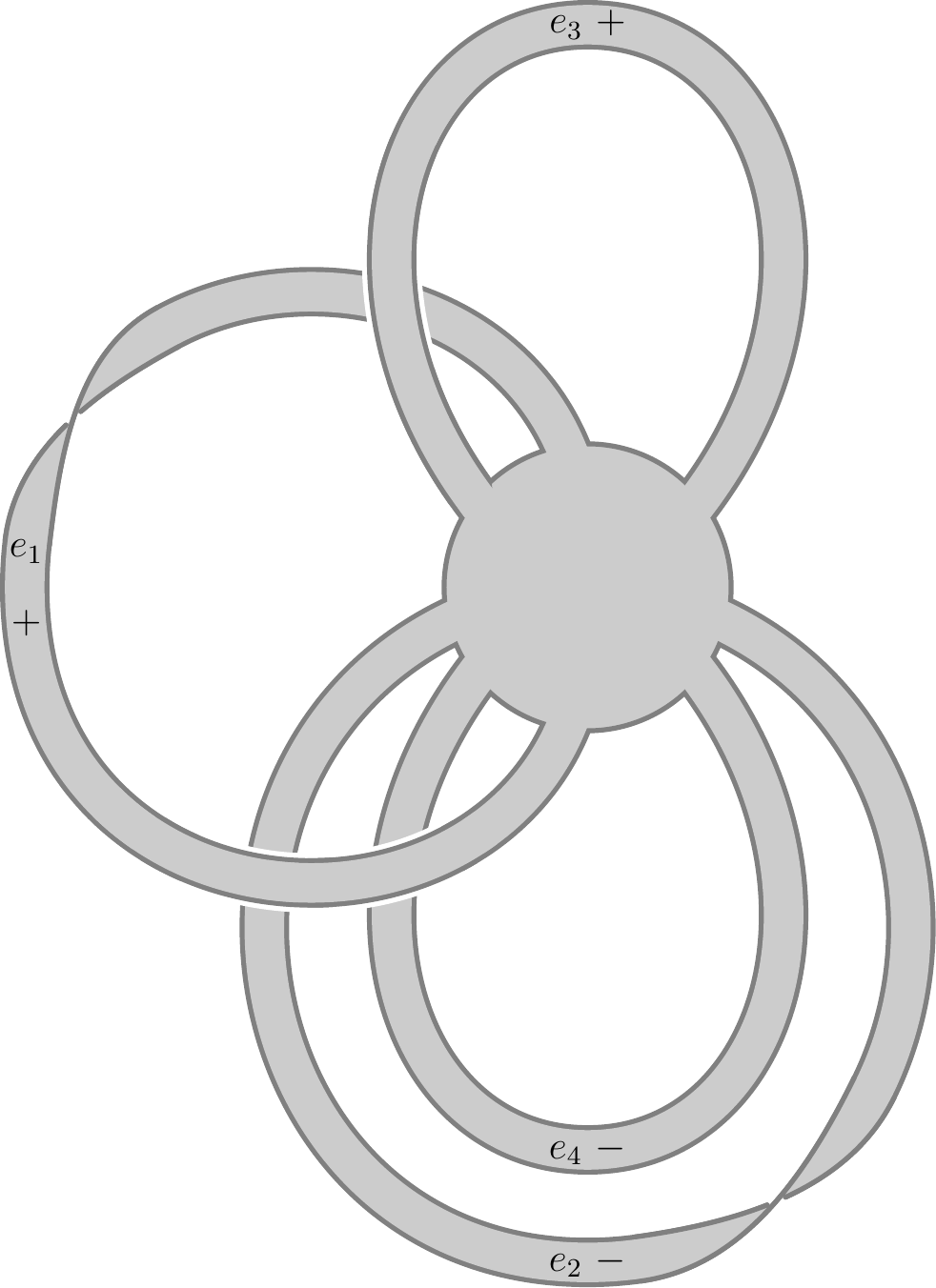}}&$\emptyset$,
        $\set{e_{1}}$, $\set{e_{2},e_{3},e_{4}}$,
        $\emptyset$&$xy^{2}z^{3}$&\raisebox{-.5\height}{\includegraphics[scale=.7]{QTree-fig-13.pdf}}&$1$
        \label{tab:QTreeExpEx}
      \end{longtable}
    \end{landscape}

\subsection{The \Kb{}}
\label{sec:kb}

We exemplify here the connected state expansion of the virtual version of the Whitehead link of \cref{VirtWH}. We label the crossings $1,2$ and $3$ as in figures \cref{PlaceArrows} and \cref{fig:ConnectedStateEx}. We choose the following order: $1\prec 2\prec 3$. On one hand,
\begin{align}
  [L](A,B,d)=&\sum_{\text{states $\ks$ of $L$}}A^{a_{L}(\ks)}B^{b_{L}(\ks)}d^{c_{L}(\ks)-1}\\
  =&A^{3}d+3A^{2}B+2AB^{2}+AB^{2}d+B^{3}.\label{eq:StateExpEx}
\end{align}
On the other hand,
\begin{align}
  &\sum_{\substack{\text{connected}\\\text{states $\ks'$ of
        $L$}}}A^{a_{L}(\ks')}B^{b_{L}(\ks')}
  \big(1+Bd/A\big)^{|\cL_{\text{o}}^{a}(\ks')|}\big(1+Ad/B\big)^{|\cL_{\text{o}}^{b}(\ks')|}\\
  =&A^{2}B(1+Bd/A)+A^{2}B+AB^{2}+A^{2}B(1+Ad/B)+AB^{2}+B^{3},
\end{align}
which is easily checked to be equal to (\cref{eq:StateExpEx}).
\begin{longtable}{|>{
    }Sc|>{
    }Sc||>{
    }Sc|}
     \caption{States of $L$}\\
    \hline
 State $\ks$&$(a_{L}(\ks),b_{L}(\ks),c_{L}(\ks))$&\parbox{2.8cm}{\centering$\cL_{\text{o}}^{a}(\ks),\cL_{\text{o}}^{b}(\ks)$\\\text{(if $c_{L}(\ks)=1$)}}\\
    \hline
    \endhead
    \multicolumn{3}{r}{\textsf{Continued on next page}}
    \endfoot
    \hline
    \endlastfoot
    \raisebox{-.5\height}{\includegraphics[scale=.2]{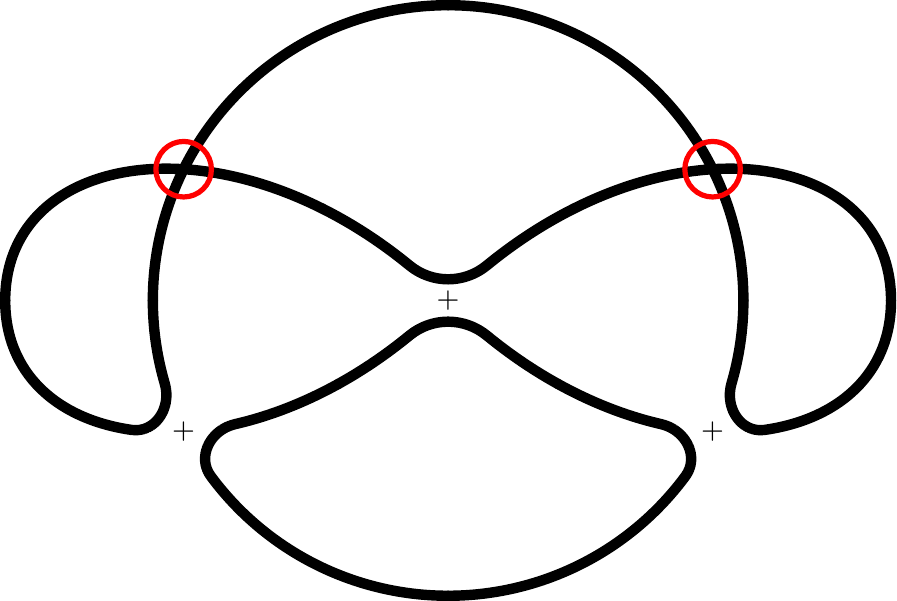}}&$(3,0,2)$&\\
    \raisebox{-.5\height}{\includegraphics[scale=.2]{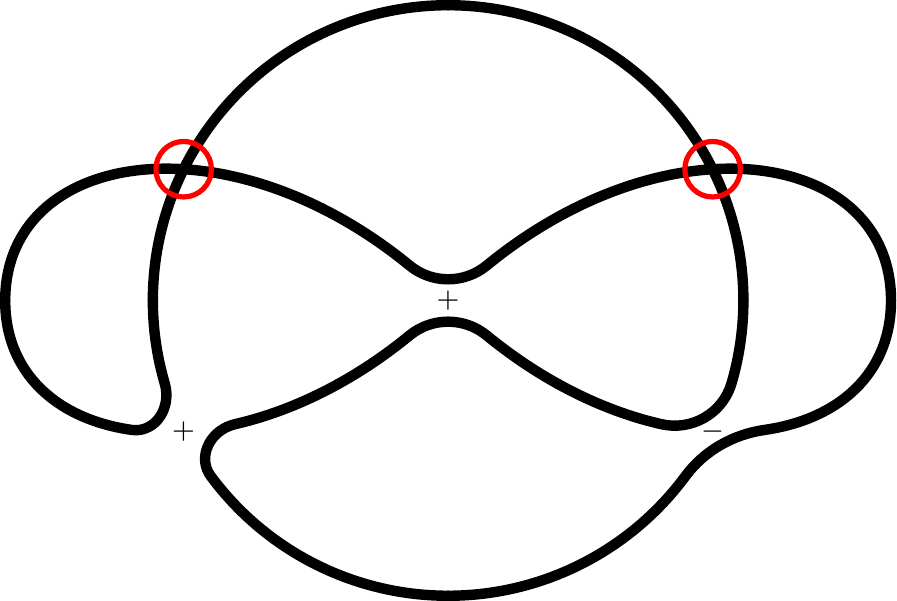}}&$(2,1,1)$&$\set{1},\emptyset$\\
    \raisebox{-.5\height}{\includegraphics[scale=.2]{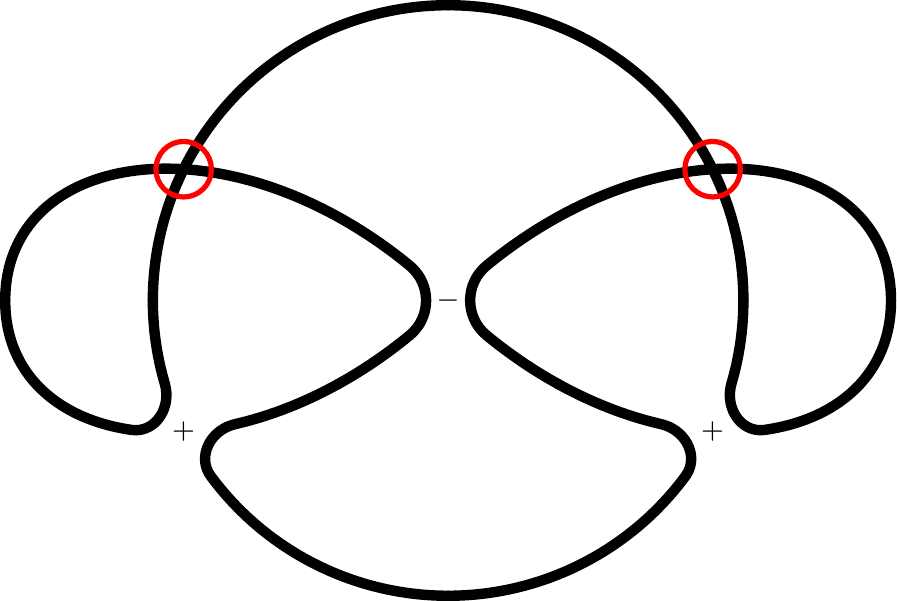}}&$(2,1,1)$&$\emptyset,\emptyset$\\
    \raisebox{-.5\height}{\includegraphics[scale=.2]{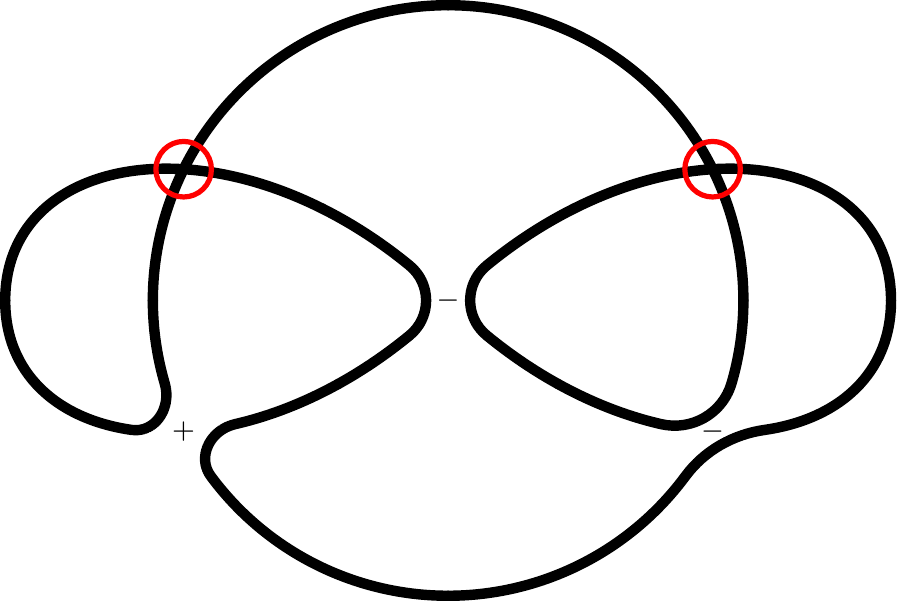}}&$(1,2,1)$&$\emptyset,\emptyset$\\
    \raisebox{-.5\height}{\includegraphics[scale=.2]{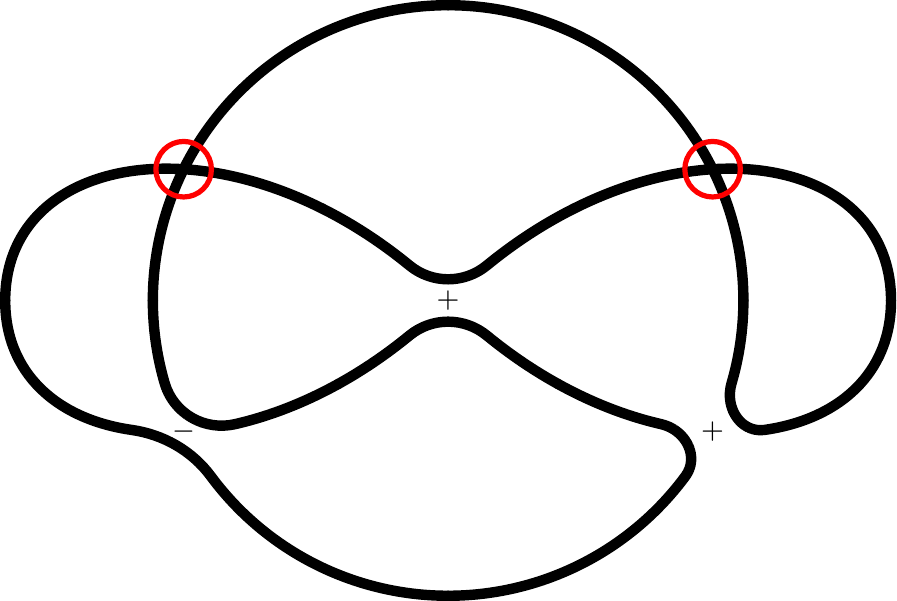}}&$(2,1,1)$&$\emptyset,\set{1}$\\
    \raisebox{-.5\height}{\includegraphics[scale=.2]{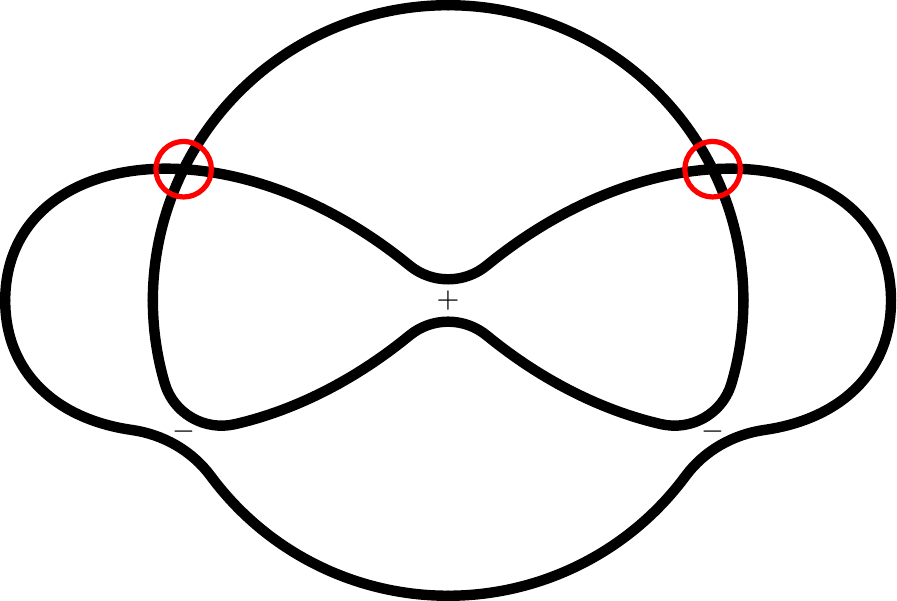}}&$(1,2,2)$&\\
    \raisebox{-.5\height}{\includegraphics[scale=.2]{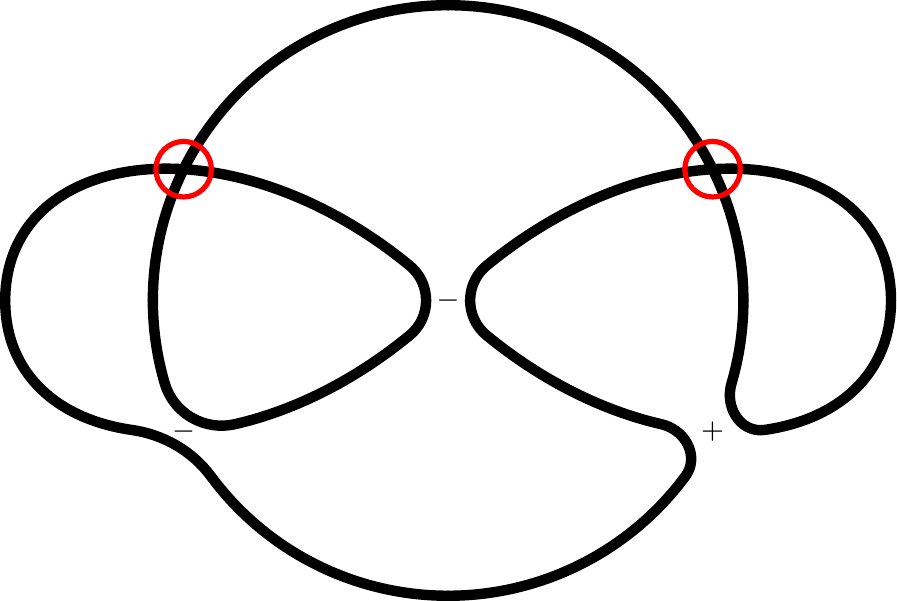}}&$(1,2,1)$&$\emptyset,\emptyset$\\
    \raisebox{-.5\height}{\includegraphics[scale=.2]{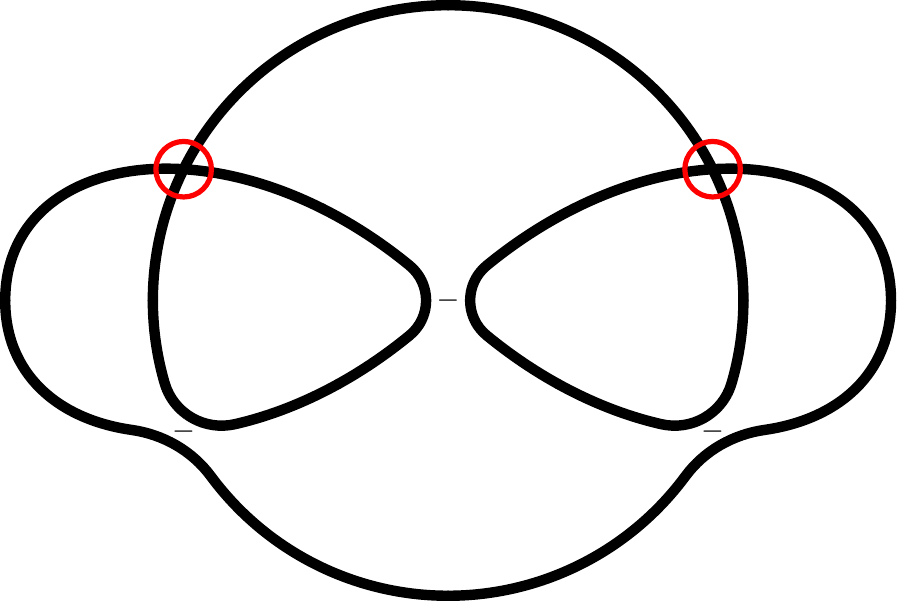}}&$(0,3,1)$&$\emptyset,\emptyset$
\end{longtable}
\newpage
\renewcommand{\bibfont}{\small}

\bibliographystyle{fababbrvnat}
\bibliography{biblio-articles,biblio-books}

\end{document}
